\documentclass[10pt]{amsart}
\usepackage[numbers,sort]{natbib}
\usepackage[margin=1.5in]{geometry}
\usepackage{xcolor}
\usepackage{xcolor}
\newtheorem{theorem}{Theorem}[section]
\newtheorem{lemma}[theorem]{Lemma}
\newtheorem{proposition}{Proposition}[section]

\theoremstyle{definition}

\let\hide\iffalse
\let\unhide\fi

\usepackage{graphicx}
\usepackage{amssymb,amsmath,amsthm,}
\usepackage{a4wide}
\usepackage{amscd}
\usepackage{amsfonts}
\usepackage{amssymb}
\usepackage{latexsym}
\usepackage{color}
\usepackage{esint}
\usepackage{graphicx}
\usepackage{caption}
\usepackage{subcaption}
\usepackage{amsthm}

\usepackage[makeroom]{cancel}
\usepackage{mathtools}

\DeclareMathOperator*{\esssup}{ess\,sup}

\let\p=\partial

\newcommand{\avv}{\langle v\rangle}

\newcommand{\R}{\mathbb{R}}

\newcommand{\be}{\begin{equation}}
\newcommand{\bm}{\begin{multline}}
\newcommand{\ee}{\end{equation}}

\newcommand{\overbar}[1]{\mkern 1.3mu\overline{\mkern-1.3mu#1\mkern-1.3mu}\mkern 1.3mu}

\newcommand{\Bes}{\begin{eqnarray*}}
	\newcommand{\Ees}{\end{eqnarray*}}
\newcommand{\Be}{\begin{equation}}
\newcommand{\Ee}{\end{equation}}

\def\MR#1{}
\def\p{\partial}

\def\R{\mathbb{R}}

\def\B{\begin{equation}}
\def\E{\end{equation}}
\def\BN{\begin{eqnarray*}}
\def\EN{\end{eqnarray*}}

\usepackage{color}

\theoremstyle{remark}

\numberwithin{equation}{section}

\allowdisplaybreaks

\begin{document}
\title[Large amplitude solution of BGK model: Relaxation to quadratic nonlinearity]{Large amplitude problem of BGK model: Relaxation to quadratic nonlinearity}

\author{Gi-Chan Bae}
\address{Research institute of Mathematics, Seoul National University, Seoul 08826, Republic of Korea }
\email{gcbae02@snu.ac.kr}

\author{Gyounghun Ko}
\address{Department of Mathematics, Pohang University of Science and Technology, South Korea}
\email{gyounghun347@postech.ac.kr}

\author{Donghyun Lee}
\address{Department of Mathematics, Pohang University of Science and Technology, South Korea}
\email{donglee@postech.ac.kr}

\author{Seok-Bae Yun}
\address{Department of mathematics, Sungkyunkwan University, Suwon 16419, Republic of Korea }
\email{sbyun01@skku.edu}




\begin{abstract}
Bhatnagar–Gross–Krook (BGK) equation is a relaxation model of the Boltzmann equation which is widely used in place of the Boltzmann equation for the simulation of various kinetic flow problems. In this work, we study the asymptotic stability of the BGK model when the initial data is not necessarily close to the global equilibrium pointwisely.
Due to the highly nonlinear structure of the relaxation operator, the argument developed to derive
the bootstrap estimate for the Boltzmann equation leads to a weaker estimate in the case of the BGK model, which does not exclude the possible blow-up of the perturbation. To overcome this issue, we carry out a refined analysis of the macroscopic fields to guarantee that the system transits from a highly nonlinear regime into a quadratic nonlinear regime after a long but finite time, in which the highly nonlinear perturbative term relaxes to essentially quadratic nonlinearity.
\end{abstract}

\maketitle
\tableofcontents
\section{Introduction}

\subsection{BGK model}
The Boltzmann equation is the fundamental equation bridging the particle description and the fluid description of gases \cite{C,CIP,Sone,Sone2}. However, the high dimensionality of the equation and the complicated structure of the collision operator have been major obstacles in applying the Boltzmann equation to various flow problems in kinetic theory. In this regard, a relaxational model equation, which now goes by the name BGK model, was introduced in pursuit of a numerically amenable model of the Boltzmann equation \cite{BGK,Wal}:
\begin{align} \label{BGK}
\begin{split}
	\partial_{t}F + v\cdot\nabla_x F &= \nu(\mathcal{M}(F) - F),\quad (t,x,v) \in \mathbb{R}^+ \times \mathbb{T}^{3}\times \mathbb{R}^3, \cr
F(0,x,v)&=F_0(x,v).
\end{split}
\end{align}
Instead of tracking the complicated collision process using the collision operator of the Boltzmann equation, the BGK model captures the relaxation process by measuring the distance between the velocity distribution to its local equilibrium state:
\begin{align*}
	\mathcal{M}(F)(t,x,v) := \frac{\rho(t,x)}{\sqrt{(2\pi T(t,x))^3}}e^{-\frac{\vert v -U(t,x)\vert^2}{2T(t,x)}}.
\end{align*}
which is called the local Maxwellian. The macroscopic density $\rho$, bulk velocity $U$, and temperature $T$ are defined by:
\begin{align} \label{macro quan}
	\begin{split}
		&\rho(t,x):= \int_{\mathbb{R}^3} F(t,x,v) dv, \\
		&\rho(t,x)U(t,x):= \int_{\mathbb{R}^3} F(t,x,v) v dv, \\
		&3\rho(t,x)T(t,x) := \int_{\mathbb{R}^3} F(t,x,v) \vert v - U(t,x) \vert^2 dv.
	\end{split}
\end{align}
Various forms are available for the collision frequency $\nu$. In this work, we consider the collision frequency of the following form:
\begin{align} \label{collision freq}
	\nu(t,x) = (\rho^aT^b)(t,x),  \quad a\geq b\geq0,
\end{align}
which covers most of the relevant models in the literature.
The relaxation operator satisfies the following cancellation property because $\mathcal{M}$ shares the first three moments with the distribution function:
\begin{align}\label{cancellation}
	\int_{\mathbb{T}^3\times\mathbb{R}^3}\big\{\mathcal{M}(F)- F \big\}\begin{pmatrix} 1\\ v\\ |v|^2 \end{pmatrix} dvdx = 0.
\end{align}
This leads to the conservation laws of mass, momentum, and energy:
\begin{align}\label{conservf}
\frac{d}{dt}\int_{\mathbb{T}^3\times\mathbb{R}^3}F(t,x,v)\begin{pmatrix} 1\\ v\\ |v|^2 \end{pmatrix} dvdx = 0,
\end{align}
and the celebrated H-theorem:
\begin{align*}
\frac{d}{dt}\int_{\mathbb{T}^3\times\mathbb{R}^3}F(t)\ln F(t) dvdx =
\int_{\mathbb{T}^3\times\mathbb{R}^3}\{\mathcal{M}(F) - F\}\{\ln\mathcal{M}(F) - \ln F \}dvdx\leq 0.
\end{align*}
The first mathematical result was obtained in \cite{Perthame}, in which Perthame obtained the global existence of weak solutions when the mass, energy, and entropy of the initial data are bounded. Perthame and Pulvirenti then found in \cite{P-P} that the existence and uniqueness are guaranteed in a class of weighted $L^{\infty}$ norm. This result was relaxed to $L^p$ setting \cite{Z-H}, extended to BGK models in the external field or mean-field \cite{Zhang} and ellipsoidal BGK \cite{Yun2}. The existence and asymptotic behavior of solutions to the BGK model near equilibrium was considered in \cite{Bello,Yun1,Yun22,Yun3}. A stationary solution for the BGK model was found using the Schauder fixed point theorem in \cite{Ukai,Nouri}. The existence and uniqueness of stationary solutions to the BGK model in a slab were investigated in \cite{Bang-Y,Brull-Y}. The argument was extended to a relativistic BGK model \cite{HY} and the quantum BGK model \cite{BGCY}. Various macroscopic limits such as the hydrodynamic limit problem and diffusion limit can be found in \cite{DMOS,M-M-M,Mellet,SR1,SR2}. For numerical studies on the BGK models, see \cite{F-J,F-R,Issau,M,M-S,RSY,RY} and references therein.\\

In this paper, we consider $L^{\infty}_{x,v}$ solution of the BGK equation. Low regularity $L^{\infty}$ solution via $L^{2}$-$L^{\infty}$ bootstrap argument was developed by Guo \cite{Guo10} to solve the Boltzmann equation with several boundary conditions. The approach was widely used and extended to solve various problems in more general boundaries and to get regularity results. We refer to \cite{KimLee, cylinder, CKLVPB, GKTT2017, GKTT2, Kim2011, CK1,KLKRM} and references therein. In these works, however, sufficiently small initial (weighted) $L^{\infty}_{x,v}$ data had to be imposed to obtain global well-posedness and convergence to equilibrium.  \\
\indent Restriction to small $L^{\infty}_{x,v}$ initial data for the Boltzmann equation was removed by Duan et al \cite{DHWY2017} by imposing small relative entropy and $L^p$ type smallness for initial data in \cite{DHWY2017, DW2019}. This type of problem is usually called large amplitude problem, because it allows initial data to be far from the global equilibrium pointwisely. This argument has been further developed into several boundary condition problems and polynomial tail in large velocity. See \cite{DW2019, KLP2022, DKL2020, CaoJFA}. \\
\indent Meanwhile, to the best of the author's knowledge, there has been only one result regarding the large amplitude problem of the BGK model \cite{DWY2017}. Due to the strong nonlinearity of the relaxation operator of the BGK model, the authors introduced an additional condition which means that the initial data remains close to global Maxwellian $\mu$ in a weighted $L^{1}$ norm along the characteristic, as well as small relative entropy. Moreover, the asymptotic behavior was not obtained in \cite{DWY2017}.  \\
\indent In this paper, we remove the additional initial condition imposed in \cite{DWY2017} by performing some refined controls of macroscopic fields which guarantee that the system transits into the quadratic nonlinear regime where we can use bootstrap argument to prove the global existence and convergence to the equilibrium. \\

\subsection{Main theorem and scheme of proof}
Let us write $F = \mu + f$ where $\mu=\mu(v)$ is the global equilibrium:
\begin{align*}
\mu(v)=\frac{1}{\sqrt{(2\pi)^3}}e^{-\frac{|v|^2}{2}},
\end{align*}
and $f$ denotes the perturbation around the equilibrium. In terms of $f$, BGK equation \eqref{BGK} can be rewritten as
\Be \label{expan}
\begin{split}
	\p_{t}f + v\cdot\nabla_x f  &= Lf + \Gamma(f),
\end{split}
\Ee
where $L$ denotes the linearized relaxation operator and $\Gamma$ is nonlinear perturbation.
For the derivation \eqref{expan} and explicit form of $L$ and $\Gamma$, see Lemma \ref{linearize}.
\hide
where characteristic becomes
\[
X(s;t,x,v) = x - v(t-x), \quad V(s;t,x,v) = v. \\
\]
\unhide

To state our main theorem, we need to define relative entropy:
\begin{align}\label{relaentp}
\mathcal{E}(F)(t)=\int_{\mathbb{T}^3\times\mathbb{R}^3} (F\ln F -\mu \ln \mu) dvdx= \int_{\mathbb{T}^3\times\mathbb{R}^3} F\ln \frac{F}{\mu} dvdx,
\end{align}
where the last equality comes from \eqref{conservf}.
We also introduce some necessary notations:
\begin{itemize}
\item We define $q$-th order velocity weight as $\langle v \rangle^q:=1+|v|^q$.
\item We define the standard $L^\infty$ norm
\begin{align*}
\|F(t)\|_{L^{\infty}_{x,v}} := \esssup_{(x,v)\in\mathbb{T}^3\times\mathbb{R}^3} |F(t,x,v)|,
\end{align*}
and weighted $L^{\infty}$ norm as
\begin{align*}
\|F(t)\|_{L^{\infty,q}_{x,v}} := \esssup_{(x,v)\in\mathbb{T}^3\times\mathbb{R}^3} \langle v \rangle^q|F(t,x,v)| , \quad \|F(t)\|_{L^{\infty}_{x,v}(m)} := \esssup_{(x,v)\in\mathbb{T}^3\times\mathbb{R}^3} m(v)|F(t,x,v)| .
\end{align*}
\item We denote standard $L^2$ norm
\begin{align*}
\|F\|_{L^{2}_v} := \left(\int_{\mathbb{R}^3} |F(v)|^2 dv \right)^{\frac{1}{2}}, \quad \|F\|_{L^{2}_{x,v}} := \left(\int_{\mathbb{T}^3\times \mathbb{R}^3} |F(x,v)|^2 dvdx \right)^{\frac{1}{2}},
\end{align*}
and weighted $L^2$ norm as
\begin{align*}
\|F\|_{L^{2}_v(m)} := \left(\int_{\mathbb{R}^3} |m(v)F(v)|^2 dv \right)^{\frac{1}{2}}, \quad \|F\|_{L^{2}_{x,v}(m)} := \left(\int_{\mathbb{T}^3\times \mathbb{R}^3} |m(v)F(x,v)|^2 dvdx \right)^{\frac{1}{2}}.
\end{align*}
\item We define the pairing $\langle \cdot, \cdot \rangle_v$ as
\begin{align}
\langle g, h \rangle_v:= \int_{\mathbb{R}^3} g(v)h(v) dv, \quad \text{ if }~ gh \in L^1(\R^3_{v}). \label{bracket}
\end{align}
\end{itemize}

We are now ready to state our main theorem.
\begin{theorem}\label{maintheorem} Let $F_0$ is non-negative: $F_0(x,v)=\mu(v)+f_0(x,v)\geq 0$, and satisfies
\begin{align} \label{initial LB}
\inf_{(t,x) \in [0,\infty) \times \mathbb{T}^3}\int_{\mathbb{R}^3} F_0(x-vt,v)dv \geq C_0,
\end{align}
for some positive constant $C_0>0$. We also assume $F_0$ shares the same mass, momentum, and energy with $\mu$:
\begin{align}	\label{normal}
\int_{\mathbb{T}^{3}\times \R^{3}} f_0
(1 ,  v , |v|^{2})dvdx = 0.
\end{align}
Then, for any $M_0>0$, there exists $\varepsilon = \varepsilon(M_0, C_0)$ such that if initial data $f_0$ satisfies $(q > 10)$
\begin{align*}
\Vert f_0 \Vert_{L^{\infty,q}_{x,v}} \leq M_0, \quad \mathcal{E}(F_0) \leq \varepsilon,
\end{align*}
then there exists a unique global-in-time solution $F(t,x,v)=\mu(v) + f(t,x,v)$ to the BGK model \eqref{expan} with collision frequency \eqref{collision freq}. Moreover $f$ satisfies
\begin{align*}
\Vert  f(t) \Vert_{L^{\infty,q}_{x,v}} \leq C_{M_0} e^{-kt},
\end{align*}
where $C_{M_0}$ depending on $M_0$, and $k$ are positive constants.
\end{theorem}

In previous results on the BGK model near the global equilibrium \cite{Bello,Yun1,Yun22,Yun3}, the fact that the macroscopic fields $(\rho, U, T)$ remains close to those of global equilibrium $(1,0,1)$ was crucially used   to close energy estimates and derive asymptotic stability. Even for the large amplitude setting, developed in \cite{DWY2017}, conditions that corresponds more or less to  the statement that the initial macroscopic fields lie close to the global equilibrium has to be imposed to control nonlinear terms and derive bootstrap estimates. \\
\begin{figure}
\includegraphics[width=15cm]{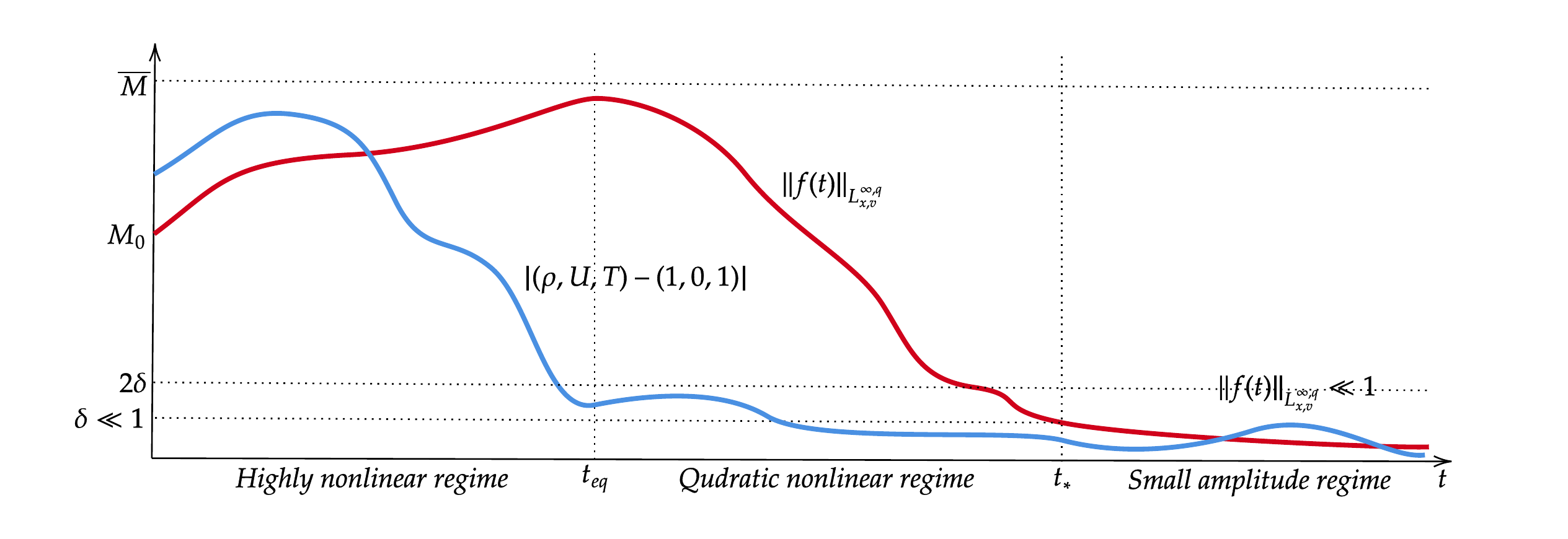}
\caption{Three different regimes}
\label{fig}
\end{figure}
\indent To overcome this restriction, the relaxation of macroscopic fields into the equilibrium macroscopic fields $(1,0,1)$ has to be carefully investigated. In this regard, we divide the evolution of the solution into three different phases, namely, highly nonlinear regime, quadratic nonlinear regime, and small amplitude regime (See Figure \ref{fig}). In the highly nonlinear regime, the amplitude of the perturbation can be arbitrarily large and its macroscopic fields are not necessarily close to $(1,0,1)$. We show that macroscopic fields relax to $(1,0,1)$ uniformly after a time $t_{eq}$ which is the onset of the quadratic nonlinear regime. This enables us to carry out the crucial bootstrap argument. After a sufficiently large time $t_{*}$, the solution enters the near-equilibrium regime (small amplitude regime) for which various existence theories are available. \\
\indent The major difficulties arise in the first two regimes in Figure \ref{fig} in which the amplitude of solutions is not necessarily small. In the study of the large-amplitude solution to the Boltzmann equation \cite{DHWY2017,DW2019,DKL2020,KLP2022}, the fact that the perturbation term is only quadratically nonlinear was crucially used in the derivation of the following key bootstrap estimate:
 \begin{align*}
 	\Vert  f(t) \Vert_{L^{\infty,q}_{x,v}} \lesssim C(\Vert  f_0 \Vert_{L^{\infty,q}_{x,v}})\left(1+\int_0^t \Vert f(s)\Vert_{L^{\infty,q}_{x,v}}ds \right)e^{-\lambda t} +  D, \quad D \ll 1.
 \end{align*}
In the case of the BGK model, the typical term in the nonlinear perturbation $\Gamma(f)$ looks like  
\begin{align} \label{NL terms}
	\sum_{1\leq i,j\leq 5} \mathcal{A}_{ij}(\rho,(v-U),U,T)\mathcal{M}(F)  \int_{\mathbb{R}^3}f e_i dv \int_{\mathbb{R}^3}f e_jdv,
\end{align}
where $\mathcal{A}_{ij}$ ($1\leq i,j \leq 5$) denote generic rational functions and $e_i ~(i=1,\dots,5)$ are the orthogonal basis of the null space of $L$, and suffers from much stronger nonlinearity than the quadratic nonlinearity of the Boltzmann equation.
In the presence of such strong nonlinearity, a naive computation would lead to the following estimate:
\begin{align*}
 	\Vert  f(t) \Vert_{L^{\infty,q}_{x,v}} \lesssim C(\Vert  f_0 \Vert_{L^{\infty,q}_{x,v}})\left(1+\int_0^t \left[ \Vert f(s)\Vert_{L^{\infty,q}_{x,v}}\right]^nds \right)e^{-\lambda t} +  D,
 \end{align*}
where exponent $n$ is determined by the order of nonlinearity of $\Gamma$. Unfortunately, this does not exclude the possibility of a blow-up of $\Vert  f(t)\Vert_{L^{\infty,q}_{x,v}}$ and, therefore, cannot be applied to bootstrap arguments.
To overcome this difficulty, we first note that the strong nonlinearity of \eqref{NL terms} comes from the nonlinearity of $\mathcal{A}_{ij}$. We will perform careful asymptotic analysis of $\rho,U,T$ to show that after some finite time $t_{eq}$ the nonlinearity of $\mathcal{A}_{ij}$ essentially vanishes so that \eqref{NL terms} become essential quadratically nonlinear.
This enables one to derive the desired bootstrap inequality with $n=1$ in the quadratic nonlinear regime.

This paper is organized as follows. In Section 2, we consider the linearization of the BGK model and derive basic estimates for the macroscopic fields. In Section 3, under a priori assumption, we prove the key estimate to control the macroscopic fields. Especially, we obtain the transition time $t_{eq}$ after which the solution enters the quadratic nonlinear regime. In Section 4, we prove the local-in-time existence and uniqueness of the BGK solutions. In Section 5, we show that the solution satisfies the desired bootstrap inequality in the quadratic nonlinear regime. In Section 6, we prove the well-posedness and exponential decay of the solution to the BGK equation in the small amplitude regime. In Appendix \ref{Nonlincomp}, we present the explicit form of the nonlinear perturbation $\Gamma$.\\

\section{Linearization and basic estimates}
In this section, we recall the linearization of the BGK model \eqref{BGK}, basic estimates for macroscopic fields. 

\begin{lemma}\emph{\cite{BY2023,Yun1}} \label{linearize} Let $F=\mu+f$. Then the BGK model \eqref{BGK} can be rewritten in terms of $f$ as follows:
\begin{align}\label{reform BGK}
\p_{t}f + v\cdot\nabla_x f  + f &= \mathbf{P}f + \Gamma(f).
\end{align}
The linear term $\mathbf{P}f$ is defined as
\begin{align}\label{pf}
\mathbf{P}f&=\int_{\mathbb{R}^3}fdv \mu+\int_{\mathbb{R}^3}fvdv\cdot (v\mu) +\int_{\mathbb{R}^3}f\frac{|v|^2-3}{\sqrt{6}}dv\left(\frac{|v|^2-3}{\sqrt{6}}\mu\right)=\sum_{i=1}^5 \langle f,e_i\rangle_v (e_i \mu),
\end{align}
where we used the pairing notation $\langle \cdot , \cdot \rangle_v$ in \eqref{bracket} and $(e_1,\cdots,e_5)=(1,v,(\vert v \vert^2-3)/\sqrt{6})$. The nonlinear term $\Gamma(f)$ is written as
\begin{align} \label{def gamma}
\Gamma(f) &= \Gamma_1(f) + \Gamma_2(f).
\end{align}
Here, 
\begin{align*}
\Gamma_1(f) &= (\mathbf{P}f-f) \sum_{1\leq i \leq 5}\int_0^1 A_i(\theta) d\theta \langle f,e_i \rangle_{v},
\end{align*}
(precise definition of $A_{i}(\theta)$ is given in \eqref{A12}) and
\begin{align}\label{polyform}
\begin{split}
\Gamma_2(f) &= \rho^a T^b\sum_{1\leq i,j\leq 5}\int_0^1 \left[\nabla_{(\rho_{\theta},\rho_{\theta} U_{\theta}, G_{\theta})}^2 \mathcal{M}(\theta)\right]_{ij}  (1-\theta) d\theta \langle f,e_i \rangle_{v} \langle f,e_j \rangle_{v} \cr
&= \rho^a T^b\sum_{1\leq i,j\leq 5}\int_0^1 \frac{\mathcal{P}_{ij}((v-U_{\theta}),U_{\theta},T_{\theta})}{\rho_{\theta}^{\alpha_{ij}}T_{\theta}^{\beta_{ij}}}\mathcal{M}(\theta)  (1-\theta) d\theta \int_{\mathbb{R}^3}f e_i dv \int_{\mathbb{R}^3}f e_jdv, 
\end{split}
\end{align}
where the transition of the macroscopic fields are defined as
\begin{align}\label{transition}
\rho_{\theta} = \theta \rho + (1-\theta), \qquad
\rho_{\theta} U_{\theta} = \theta \rho U, \qquad
\rho_{\theta}|U_{\theta}|^2+3\rho_{\theta} T_{\theta} - 3\rho_{\theta} = \theta (\rho|U|^2+3\rho T - 3\rho).
\end{align}
Here, $\mathcal{P}_{ij}$ denotes a generic polynomial such that $\mathcal{P}_{ij}(x_1,\cdots,x_n)=\sum_m a^{ij}_m x_1^{m_1}\cdots x_n^{m_n}$ and $\alpha_{ij},\beta_{ij} \geq 0$.
Precise definitions of $\alpha_{ij},\beta_{ij}$, and $\mathcal{P}_{ij}$ are given in Appendix \ref{Nonlincomp}.
\end{lemma}
\begin{proof}
The linearization of the local Maxwellian $\mathcal{M}(F)$ and $\nu(t,x)$ can be found in \cite{Yun1} and \cite{BY2023}, respectively. But for the reader's convenience, we briefly sketch the proof here. The main idea of linearization is constructing a convex combination of the following macroscopic fields:
\begin{align}\label{rhoUG}
\rho = \int_{\mathbb{R}^3} \mu+f dv, \quad
\rho U = \int_{\mathbb{R}^3} vf dv, \quad
G = \frac{\rho|U|^2+3\rho T - 3\rho}{\sqrt{6}}=\int_{\mathbb{R}^3} \frac{|v|^2-3}{\sqrt{6}}f dv.
\end{align}
We note that the mapping of the macroscopic fields $(\rho,U,T)\leftrightarrow (\rho,\rho U,G)$ is one to one if $\rho>0$ because of the following reverse relation:
\begin{align*}
U = \frac{\rho U}{\rho},\qquad T 
= \sqrt{\frac{2}{3}}\frac{G}{\rho} - \frac{|\rho U|^2}{3\rho^2}+1.
\end{align*}
Using the transition of the macroscopic fields \eqref{transition}, we write the local Maxwellian depending on $(\rho_{\theta},U_{\theta},T_{\theta})$ as $\mathcal{M}(\theta)$ and we apply Taylor's theorem at $\theta=0$.
\begin{align*}
\mathcal{M}(1)= \mathcal{M}(0) + \frac{d \mathcal{M}(\theta)}{d\theta}\bigg|_{\theta=0} + \int_0^1 \frac{d^2 \mathcal{M}(\theta)}{d\theta^2}(1-\theta)d\theta .
\end{align*}
Since $(\rho_{\theta},U_{\theta},T_{\theta})|_{\theta=1}=(\rho,U,T)$ and $(\rho_{\theta},U_{\theta},T_{\theta})|_{\theta=0}=(1,0,1)$, we have $\mathcal{M}(1)=\mathcal{M}(F)$ and $\mathcal{M}(0)=\mu$, respectively. Then we consider the first derivative of $\mathcal{M}(\theta)$:
\begin{align*}
\frac{d \mathcal{M}(\theta)}{d\theta}\bigg|_{\theta=0} &= \frac{d \rho_{\theta}}{d \theta}\frac{\partial  \mathcal{M}(\theta)}{\partial \rho_{\theta}}+\frac{d(\rho_{\theta}U_{\theta})}{d\theta}\frac{\partial \mathcal{M}(\theta)}{\partial (\rho_{\theta}U_{\theta})}+\frac{dG_{\theta}}{d\theta}\frac{\partial \mathcal{M}(\theta)}{\partial G_{\theta}} \cr
&= \left(\frac{d (\rho_{\theta}, \rho_{\theta} U_{\theta}, G_{\theta})}{d \theta}\right)^{T} \left( \frac{\partial(\rho_{\theta},\rho_{\theta} U_{\theta},G_{\theta})} {\partial(\rho_{\theta},U_{\theta},T_{\theta})} \right)^{-1}\nabla_{(\rho_{\theta},U_{\theta},T_{\theta})}\mathcal{M}(\theta)\bigg|_{\theta=0},
\end{align*}
where we used that the last definition of \eqref{transition} is equivalent to $G_{\theta}=\theta G$. Then substituting the computation of the Jacobian and $\nabla_{(\rho_{\theta},U_{\theta},T_{\theta})}\mathcal{M}(\theta)$ in Lemma \ref{Jaco} at $\theta=0$ with
\begin{align*}
\left(\frac{d (\rho_{\theta}, \rho_{\theta} U_{\theta}, G_{\theta})}{d \theta}\right) = \left(\int_{\mathbb{R}^3} f dv, \int_{\mathbb{R}^3} vf dv, \int_{\mathbb{R}^3} \frac{|v|^2-3}{\sqrt{6}}f dv\right),
\end{align*}
we obtain that
\begin{align*}
\frac{d \mathcal{M}(\theta)}{d\theta}\bigg|_{\theta=0}= \mathbf{P}f.
\end{align*}
For the nonlinear term, applying the chain rule twice yields
\begin{align*}
\frac{d^2\mathcal{M}(\theta)}{d\theta^2}
&=(\rho-1, \rho U, G)^T\left\{D^2_{(\rho_{\theta},\rho_{\theta} U_{\theta},G_{\theta})}\mathcal{M}(\theta)\right\}(\rho-1, \rho U, G) \cr
&= \left[\nabla_{(\rho_{\theta},\rho_{\theta} U_{\theta}, G_{\theta})}^2 \mathcal{M}(\theta)\right]_{ij} \langle f,e_i \rangle_{v} \langle f,e_j \rangle_{v}.
\end{align*}
The explicit form of $\nabla_{(\rho_{\theta},\rho_{\theta} U_{\theta}, G_{\theta})}^2 \mathcal{M}(\theta)$ will be given in Appendix \ref{Nonlincomp}. Now we consider the collision frequency $\nu = \rho^aT^b$. We define $\nu(\theta)= \rho_{\theta}^a T_{\theta}^b$, and by Taylor's theorem, we have
\begin{align*}
\nu(1) = \nu(0) + \int_0^1 \frac{d}{d\theta}\nu(\theta) d\theta.
\end{align*}
By an explicit computation, we have
\begin{align*}
\nu(t,x) 
&= 1+ \sum_{1\leq i \leq 5}\int_0^1 A_i(\theta) d\theta \langle f,e_i \rangle_{v},
\end{align*}
where
\begin{align} \label{A12}
\begin{split}
A_1(\theta)&=a\rho^{a-1}_{\theta} T^b_{\theta}, \qquad A_{i+1}(\theta)= - a\rho^{a-2}_{\theta} U_{\theta i} T^b_{\theta} , \quad i=1,2,3,  \cr
A_5(\theta)&= \frac{|U_{\theta}|^2-3T_{\theta}+3}{3\rho_{\theta}}a\rho^{a-1}_{\theta} T^b_{\theta} + \sqrt{\frac{2}{3}} b\rho^{a-1}_{\theta} T^{b-1}_{\theta}.
\end{split}
\end{align}
Therefore, we obtain
\begin{align*}
\nu(\mathcal{M}(F)-F)= (\mathbf{P}f - f) + \Gamma_1(f) + \Gamma_2(f).
\end{align*}
\end{proof}


\begin{lemma}\label{rho,T esti} Recall the macroscopic fields defined in \eqref{macro quan}. Let $(t,x)\in \R^+ \times \mathbb{T}^3$ and $q>5$. Then, we have the upper bounds for macroscopic fields:
\begin{align*}
\left(\begin{array}{ccc} \rho(t,x) \\ \rho(t,x) U(t,x) \\3\rho(t,x) T(t,x) + \rho(t,x)|U(t,x)|^2  \end{array}\right) = \int_{\mathbb{R}^3} F(t,x,v) \left(\begin{array}{ccc} 1 \\ v \\ |v|^2 \end{array}\right) dv \leq C_{q}  \sup_{0\leq s \leq t}\| F(s)\|_{L^{\infty,q}_{x,v}},  
\end{align*}
where $C_q = \frac{1}{5}+\frac{1}{q-5}$.
\end{lemma}
\begin{proof} We only consider the last inequality since the other inequalities are similar. Note that
\begin{align*}
\int_{\mathbb{R}^3}|v|^2F dv = \int_{\mathbb{R}^3}\frac{\langle v \rangle^q}{\langle v \rangle^q}|v|^2F dv\leq \int_{\mathbb{R}^3}\frac{1}{\langle v \rangle^q}|v|^2 dv \| F(s)\|_{L^{\infty,q}_v}.
\end{align*}
Then by the following explicit computation
\begin{align*}
\int_{\mathbb{R}^3}|v|^2\langle v \rangle^{-q} dv = 4\pi \int_0^{\infty} \frac{|v|^4}{1+|v|^q} d|v|
&\leq 4\pi\left(\int_0^1 |v|^4 d|v|+\int_1^{\infty} |v|^{4-q} d|v|\right) \leq \frac{1}{5}+\frac{1}{q-5},
\end{align*}
we obtain desired result.
\end{proof}

We present some $L^\infty$ estimates for the macroscopic fields:
\begin{lemma}\label{Pesti}\emph{\cite{P-P}} Consider a non-negative function $F \in L^{\infty}_x(\mathbb{T}^3)L^1_v(\mathbb{R}^3)$ and recall corresponding macroscopic fields defined in \eqref{macro quan}. The macroscopic fields enjoy the following estimates:
\begin{align*}
(1)~&\frac{\rho}{T^{\frac{3}{2}}} \leq C\|F\|_{L^{\infty}_{x,v}}, \cr
(2)~&|v|^q \mathcal{M}(F) \leq C_q\| F \|_{L^{\infty,q}_{x,v}}, \qquad  \textit{for}\qquad   q=0, \quad 5<q.
\end{align*}
\end{lemma}
\begin{proof}
We refer to \cite{P-P}.
\end{proof}

\section{Transition to quadratic nonlinear regime}
As mentioned before, the highly nonlinear behavior of the BGK operator combined with large amplitude $\|f(s)\|_{L_{x,v}^{\infty, q}}$ is one of the main obstacles in proving the asymptotic behavior of the solution. In this section, we prove that the macroscopic fields $(\rho,U,T)$ are uniformly close to $(1,0,1)$ for $t \geq t_{eq}$ so that $\Gamma$ of \eqref{def gamma} becomes quadratic nonlinear in terms of $f$. We also note that $t_{eq}$ should be chosen so that it depends only on initial data and other generic constants.  \\
Throughout this section, we impose a priori assumption
\begin{align} \label{Assumption}
\sup_{0\leq s \leq t_*} \| f(s)\|_{L^{\infty,q}_{x,v}} \leq \overbar{M}, \quad \mbox{for} \quad q>10,
\end{align}
where $t_{*}$ is arbitrary large as much as needed. Both $\overbar{M}$ and $t_{*}$ will be chosen depending on initial data in the proof of Theorem \ref{maintheorem} in Section \ref{sec 5}.


\begin{proposition}\label{macrodecay} Let us assume \eqref{initial LB} and the a priori assumption \eqref{Assumption}.
For any given positive constant $\delta\in(0,1)$, there exists sufficiently small $\mathcal{E}(\overbar{M},t_*,\delta)$ such that if $\mathcal{E}(F_0) \leq \mathcal{E}(\overbar{M},t_*,\delta)$, then the following estimate holds on $t\in[0,t_*]$:
\Be \label{v2f}
\left \vert \int_{\mathbb{R}^3} \langle v\rangle^2f(t,x,v) dv  \right \vert \leq C_qM_0e^{-t}  + \frac{3}{4}\delta.
\Ee
\end{proposition}
The above proposition implies the following two important properties. First, the macroscopic fields $(\rho,U,T)$ become uniformly close to $(1,0,1)$ after time $t \geq t_{eq}(M_0, \delta)$, i.e., 
\begin{align*}
|\rho-1|,~|\rho U|,~|3\rho T+\rho|U|^2-3| \leq \delta, \quad \mbox{for} \quad t\geq t_{eq}.
\end{align*}
This will be proved in Lemma \ref{macrodelta}. And after $t_{eq}$, the high order nonlinearity of $\Gamma$ on the BGK model is transformed into quadratic nonlinear form for which we are able to prove the asymptotic behavior of the solution.
To prove Proposition \ref{macrodecay}, we first decompose the L.H.S of \eqref{v2f} into several pieces.

\begin{lemma}\label{termdivide} Let us assume the a priori assumption \eqref{Assumption}. For an arbitrary real number $N >1$, we have
\begin{align*}
\left \vert \int_{\mathbb{R}^3} \langle v\rangle^2f(t,x,v) dv  \right \vert &\leq C_{q}M_0e^{-t} + \frac{C_q}{N^{q-5}} \overbar{M}  \cr
&\quad + \int_{0}^{t} e^{-(t-s)}\int_{\vert v \vert \leq N} \langle v\rangle^2\vert \mathbf{P}f (s, x - v(t-s),v)\vert dvds   \cr
&\quad + \int_{0}^{t} e^{-(t-s)} \int_{\vert v \vert \leq N} \langle v\rangle^2\vert \Gamma(f) (s, x - v(t-s),v)\vert dvds.
\end{align*}
\end{lemma}
\begin{proof}
We split the velocity integration into $\{\vert v \vert \leq N\}$ and $\{\vert v \vert \geq N\}$ for arbitrary real number $N >1 $:
\begin{align*}
\int_{\mathbb{R}^3} \langle v\rangle^2f(t,x,v) dv = \int_{\vert v \vert \geq N} \langle v\rangle^2f(t,x,v)dv + \int_{\vert v \vert \leq N} \langle v\rangle^2f(t,x,v) dv.
\end{align*}
For the large velocity region $ \{\vert v \vert \geq N\}$, applying the a priori bound $\sup_{0\leq s \leq t_*}\| f(s)\|_{L^{\infty,q}_{x,v}} \leq \overbar{M}$, we have
\begin{align*}
\left \vert \int_{\vert v \vert \geq N} \langle v\rangle^2f(t,x,v)dv\right \vert &\leq \int_{\vert v \vert \geq N} \langle v \rangle^{-q+2}\langle v \rangle^{q} \vert f(t,x,v) \vert dv 
\leq \frac{C_q}{N^{q-5}} \overbar{M},
\end{align*}
for $t\in[0,t_*]$. 
For the bounded velocity region $ \{\vert v \vert \leq N\}$, we use the mild formulation of the reformulated BGK equation \eqref{reform BGK} to get
\begin{align}\label{Rf est 2}
\begin{split}
\left \vert  \int_{\vert v \vert \leq N} \langle v\rangle^2f(t,x,v) dv \right \vert&\leq e^{-t} \int_{\vert v\vert \leq N} \langle v\rangle^2 \vert f_0 (x-vt, v)\vert dv \cr
&\quad + \int_{0}^{t} e^{-(t-s)}\int_{\vert v \vert \leq N} \langle v\rangle^2\vert \mathbf{P}f (s, x - v(t-s),v)\vert dvds   \\
&\quad + \int_{0}^{t} e^{-(t-s)} \int_{\vert v \vert \leq N} \langle v\rangle^2\vert \Gamma(f) (s, x - v(t-s),v)\vert dvds.
\end{split}
\end{align}
The first term on the R.H.S of \eqref{Rf est 2} is bounded as follows:
\begin{align*}
e^{-t}\int_{\vert v\vert \leq N} \langle v\rangle^2 \vert f_0 (x-vt, v)\vert dv = e^{-t} \int_{\vert v\vert \leq N}  \langle v \rangle^{-q+2}\langle v \rangle^{q}\vert f_0 (x-vt, v)\vert dv\leq C_{q}e^{-t} M_0,
\end{align*}
which gives the desired result.
\hide where
\begin{align*}
C_q= \int_{\vert v\vert \leq N}\langle v \rangle^{-q+2} dv = 4\pi \int_0^N \frac{|v|^2}{1+|v|^{q-2}} d|v|
&\leq 4\pi\left(\int_0^1 |v|^2 d|v|+\int_1^N |v|^{4-q} d|v|\right) \cr
&\leq \frac{1}{3}+\frac{1}{q-4}\left(1-\frac{1}{N^{q-5}}\right)
\end{align*} \unhide
\end{proof}
The estimates for the second and third terms on the R.H.S of \eqref{Rf est 2} will be given in Lemma \ref{Pfesti} and Lemma \ref{nonlinMbar}, respectively.

Before we present the estimate of the second term of \eqref{Rf est 2}, we note the following important property about the relative entropy. Since H-theorem holds for the BGK equation, the proof of the following Lemma is very similar to that of \cite{GuoQAM} or \cite{KLP2022}. For the convenience of readers, we provide detailed proof.
\begin{lemma} \emph{\cite{GuoQAM, DWY2017,KLP2022}}\label{L2 control}
Assume that $F(t,x,v)=\mu(v)+f(t,x,v)$ is a solution of the BGK model \eqref{BGK}. For any $t\geq 0$, we have
\begin{align*}
\int_{\mathbb{T}^3 \times \R^3} \frac{1}{4\mu(v)}\vert f(t,x,v) \vert ^2 \mathbf{1}_{\vert f(t,x,v) \vert \leq \mu(v)} dvdx + \int_{\mathbb{T}^3 \times \R^3} \frac{1}{4} \vert f(t,x,v) \vert \mathbf{1}_{\vert f(t,x,v) \vert > \mu(v)} dvdx \leq \mathcal{E}(F_0),
\end{align*}
where initial relative entropy $\mathcal{E}(F_0)$ was defined in \eqref{relaentp}.
\end{lemma}
\begin{proof} Notice that the mean value theorem gives
\begin{align*}
	F\ln F -\mu \ln \mu = (1+\ln \mu) (F-\mu) +\frac{1}{2\tilde{F}}\vert F-\mu \vert^2,
\end{align*}
where $\tilde{F}$ is between $F$ and $\mu$. If we define the function $\psi(x):= x\ln x - x+1$, we have
\begin{align*}
	\frac{1}{2\tilde{F}}\vert F-\mu \vert^2 &= F\ln F - \mu \ln \mu -(1+\ln \mu) (F-\mu)
	=\psi\left( \frac{F}{\mu}\right)\mu.
\end{align*}
Hence, one obtains that
\begin{align} \label{3.2 proof 1}
	\int_{\mathbb{T}^3 \times \R^3} \frac{1}{2\tilde{F}}\vert F-\mu \vert^2 dvdx = \int_{\mathbb{T}^3 \times \R^3} \psi\left( \frac{F}{\mu}\right)\mu dvdx.
\end{align}
For the L.H.S in \eqref{3.2 proof 1}, we divide
\begin{align*}
	1=\textbf{1}_{\vert F-\mu \vert \leq \mu} + \textbf{1}_{\vert F-\mu \vert > \mu}.
\end{align*}
On $\{ \vert F-\mu \vert > \mu\}$, we have
\begin{align*}
	\frac{ \vert F-\mu \vert}{\tilde{F}}= \frac{F-\mu}{\tilde{F}} > \frac{F-\frac{1}{2}F}{F}=\frac{1}{2},
\end{align*}
where we used the fact $F>2\mu$. On the other hand, over $\{\vert F-\mu \vert \leq \mu \}$, we obtain $0\leq F \leq 2\mu$. This implies that
\begin{align*}
	\frac{1}{\tilde{F}} \geq \frac{1}{2\mu}.
\end{align*}
Thus, it follows from \eqref{3.2 proof 1} that
\begin{align} \label{3.2 proof 2}
	\int_{\mathbb{T}^3 \times \R^3} \frac{1}{4\mu} \vert F-\mu \vert^2 \textbf{1}_{\vert F-\mu \vert \leq \mu}dvdx + \int_{\mathbb{T}^3 \times \R^3} \frac{1}{4}\vert F-\mu \vert \textbf{1}_{\vert F-\mu \vert > \mu} dvdx \leq \int_{\mathbb{T}^3 \times \R^3} \psi\left( \frac{F}{\mu}\right) \mu dvdx.
\end{align}
By $\psi'(x)=\ln x$, we can deduce from \eqref{BGK} that
\begin{align*}
	\partial_t \left[ \mu \psi\left(\frac{F}{\mu} \right) \right]+\nabla_x \cdot \left[ \mu \psi \left( \frac{F}{\mu}\right)v\right]= \nu (\mathcal{M}(F)-F)\ln \frac{F}{\mu}.
\end{align*}
By taking integration over $(x,v) \in \mathbb{T}^3 \times \R^3$, we obtain
\begin{align*}
	\frac{d}{dt} \int_{\mathbb{T}^3 \times \R^3} \psi\left( \frac{F}{\mu}\right)\mu dv dx = \int_{\mathbb{T}^3 \times \R^3} \nu(\mathcal{M}(F)-F)\ln F dvdx.
\end{align*}
Because of the following inequality
\begin{align*}
	\int_{\mathbb{T}^3 \times \R^3} \nu(\mathcal{M}(F)-F)\ln F dvdx\leq 0,
\end{align*}
we get
\begin{align} \label{3.2 proof 3}
	\int_{\mathbb{T}^3 \times \R^3}  \psi\left( \frac{F}{\mu}\right)\mu dv dx  \leq \int_{\mathbb{T}^3 \times \R^3}  \psi\left( \frac{F_0}{\mu}\right)\mu dvdx =\mathcal{E}(F_0).
\end{align}
Combining \eqref{3.2 proof 2} and \eqref{3.2 proof 3} yields that
\begin{align*}
	\int_{\mathbb{T}^3 \times \R^3} \frac{1}{4\mu} \vert F-\mu \vert^2 \textbf{1}_{\vert F-\mu \vert \leq \mu} dvdx + \int_{\mathbb{T}^3 \times \R^3} \frac{1}{4}\vert F-\mu \vert \textbf{1}_{\vert F-\mu \vert > \mu} dvdx \leq \mathcal{E}(F_0).
\end{align*}
We complete the proof of Lemma \ref{L2 control}.
\end{proof}

Now, we estimate the linear $\mathbf{P}f$ part on the R.H.S of \eqref{Rf est 2}.
\begin{lemma}\label{Pfesti} For $q>5$, if a priori assumption \eqref{Assumption} holds, then we have the following estimate
\begin{align*}
\int_{0}^{t}\int_{\vert v \vert \leq N } e^{-(t-s)} \langle v\rangle^2 \vert \mathbf{P}f (s, x - v(t-s),v)\vert dvds &\leq C_qN^5(1-e^{-\lambda}) \overbar{M}+\frac{C_q}{N^{q-10}}\overbar{M} \cr
&\quad + CN^6 (\lambda^{-2}+N^3) \left(\mathcal{E}(F_0)+N^{\frac{3}{2}}\sqrt{\mathcal{E}(F_0)} \right),
\end{align*}
for some arbitrary constants $N>1$ and $\lambda \in (0,t)$.
\end{lemma}
\begin{proof}
We split the integration region into $I$, $II_{1}$, and $II_{2}$ as follows:
\begin{align*}
\int_{0}^{t}e^{-(t-s)} \int_{\vert v \vert \leq N }  \langle v\rangle^2 \vert \mathbf{P}f (s, x - v(t-s),v)\vert dvds \leq I+II_1+II_2,
\end{align*}
where
\begin{align*}
I&= \int_{t-\lambda}^{t}e^{-(t-s)} \int_{\vert v \vert \leq N} \langle v\rangle^2 \int_{\mathbb{R}^3} \vert f(s,x-v(t-s),u)\vert(1+\vert u \vert + \vert u \vert ^2) dudvds, \cr
II_1&= \int_0^{t-\lambda}e^{-(t-s)} \int_{\vert v \vert \leq N} \langle v\rangle^2 \int_{\vert u \vert \geq 2N} \vert f(s,x-v(t-s),u)\vert(1+\vert u \vert + \vert u \vert ^2) dudvds, \cr
II_2&= \int_0^{t-\lambda}e^{-(t-s)} \int_{\vert v \vert \leq N} \langle v\rangle^2 \int_{\vert u \vert \leq 2N} \vert f(s,x-v(t-s),u)\vert(1+\vert u \vert + \vert u \vert ^2) dudvds,
\end{align*}
for a positive constant $\lambda \in (0,t) $. \newline
{ (Estimate of $I$)} Multiplying and dividing $\langle u \rangle^q$, we have
\begin{align}\label{pf esti 1}
\begin{split}
I
&\leq  \int_{t-\lambda}^t e^{-(t-s)} \int_{\vert v \vert \leq N}\langle v\rangle^2\int_{\mathbb{R}^3}  \langle u \rangle^{-q+2}\langle u \rangle^q \vert f(s,x-v(t-s),u) \vert dudvds \\
&\leq C_qN^5(1-e^{-\lambda})\sup_{0\leq s \leq t_*}\Vert f(s) \Vert_{L^{\infty,q}_{x,v}} \cr
&\leq C_qN^5(1-e^{-\lambda}) \overbar{M},
\end{split}
\end{align}
where we used $\int_{\vert v \vert \leq N}\langle v\rangle^2dv \leq CN^5$ and $\int_{\mathbb{R}^3}  \langle u \rangle^{-q+2}du \leq C_q $ for $q>5$.
\newline
{ (Estimate of $II_1$)} 
Similarly, we multiply and divide $\langle u \rangle^q$ on $II_1$:
\begin{align} \label{pf esti 2}
\begin{split}
II_1&=\int_0^{t-\lambda}e^{-(t-s)} \int_{\vert v \vert \leq N} \int_{\vert u \vert \geq 2N}  \langle v\rangle^2 \vert f(s,x-v(t-s),u) \vert (1+\vert u \vert+ \vert u \vert^2) dudvds \\
&\leq \int_0^{t-\lambda}e^{-(t-s)} ds \int_{\vert v \vert \leq N}\langle v\rangle^2 dv \int_{\vert u \vert \geq 2N} \langle u \rangle^{-q+2} du \sup_{0\leq s \leq t_*}\Vert f(s) \Vert_{L^{\infty,q}_{x,v}} \\
&\leq \frac{C_q}{N^{q-10}}\overbar{M},
\end{split}
\end{align}
where we used $\int_{\vert v \vert \leq N}\langle v\rangle^2dv \leq CN^5$ and $\int_{\vert u \vert \geq 2N} \langle u \rangle^{-q+2} du \leq CN^{-q+5} $ for $q>5$.
\newline
{ (Estimate of $II_2$)} 
Using the upper bound $\langle v\rangle^2 \leq N^2$ and $(1+|u|+|u|^2)\leq 4N^2$, we have
\begin{align*}
II_2&\leq 4N^6 \int_0^{t-\lambda}e^{-(t-s)} \int_{\vert v \vert \leq N}  \int_{\vert u \vert \leq 2N}  \vert f(s,x-v(t-s),u)\vert dudvds.
\end{align*}
Then we apply a change of variable $y=x-v(t-s)$ with $dy=-(t-s)^3 dv$ to make a change $dv$ integral to space integral $dy$. Such a change of variable transforms the integral region $\{|v|\leq N\}$ to a sphere with a center $x$ and radius $v(t-s)$. We note that the maximum radius is $N(t-s)$. Since the space variable is in the torus, if $N(t-s) \geq 1$, then the maximum number of cubic reached by $y$ is $(N(t-s))^3$.
Conversely, if $N(t-s) \leq 1$, then the minimum number of cubic reached by $y$ is $1$.
Thus we have
\begin{align*}
II_2&\leq 4N^6 \int_0^{t-\lambda}e^{-(t-s)} \frac{1+(N(t-s))^3}{(t-s)^3} \int_{\mathbb{T}^3}  \int_{\vert u \vert \leq 2N}  \vert f(s,y,u)\vert dudyds.
\end{align*}
In order to apply Lemma \ref{L2 control}, we split the integral region into $\{|f|> \mu\}$ and $\{|f|\leq \mu\}$, and we multiply  $1/\sqrt{\mu(u)} \geq 1$ on the region $\{|f|\leq\mu\}$:
\begin{align*}
II_2&\leq 4N^6 \int_0^{t-\lambda}e^{-(t-s)} \frac{1+(N(t-s))^3}{(t-s)^3} \left( \int_{\mathbb{T}^3}\int_{\vert u \vert \leq 2N} \vert f(s,y,u) \vert\mathbf{1}_{\vert f(s,y,u)\vert > \mu(u)} dudy \right. \cr
&\quad \left. +  \int_{\mathbb{T}^3}\int_{\vert u \vert \leq 2N} \frac{1}{\sqrt{\mu(u)}}\vert f(s,y,u) \vert\mathbf{1}_{\vert f(s,y,u)\vert \leq \mu(u)}  dudy \right)ds.
\end{align*}
By the H\"{o}lder inequality on the second term, and applying Lemma \ref{L2 control}, we have
\begin{align}\label{cov}
\begin{split}
II_2&\leq CN^6 \int_0^{t-\lambda}e^{-(t-s)} \frac{1+(N(t-s))^3}{(t-s)^3} \left[ \int_{\mathbb{T}^3}\int_{\vert u \vert \leq 2N} \vert f(s,y,u) \vert\mathbf{1}_{\vert f(s,y,u)\vert > \mu(u)} dudy \right. \cr
&\quad \left. +  \left(\int_{\mathbb{T}^3}\int_{\vert u \vert \leq 2N} \frac{1}{\mu(u)}\vert f(s,y,u) \vert^2\mathbf{1}_{\vert f(s,y,u)\vert \leq \mu(u)}dudy\right)^{\frac{1}{2}}\left(\int_{\mathbb{T}^3}\int_{\vert u \vert \leq 2N}1dudy \right)^{\frac{1}{2}} \right]ds \cr
&\leq CN^6 (\lambda^{-2}+N^3) \left(\mathcal{E}(F_0)+N^{\frac{3}{2}}\sqrt{\mathcal{E}(F_0)} \right),
\end{split}
\end{align}
where we used
\begin{align*}
\int_0^{t-\lambda}e^{-(t-s)} \frac{1+(N(t-s))^3}{(t-s)^3} ds = \int_{\lambda}^{t}e^{-\tau} \frac{1+(N\tau)^3}{\tau^3} d\tau \leq C(\lambda^{-2}+N^3).
\end{align*}
Combining \eqref{pf esti 1}, \eqref{pf esti 2} and \eqref{cov}, we finish the proof.
\end{proof}

The third term of the R.H.S of \eqref{Rf est 2} has the nonlinear term $\Gamma(f)$. To control the nonlinear term, we should control the macroscopic fields under a priori assumption \eqref{Assumption}.

\begin{lemma}\label{macrobarM} Assume \eqref{initial LB} and \eqref{Assumption}. Then the macroscopic fields $(\rho,U,T)$ are bounded as follows:
\hide \begin{align*}
&(1)~  \rho(t,x) \leq C_q \overbar{M}, \quad \mbox{and} \quad  \rho(t,x) \geq C_0e^{-C_q\overbar{M}^at}, \cr
&(2)~ |U(t,x)| \leq C_q\overbar{M}e^{C_q\overbar{M}^at}  \cr
&(3)~  T(t,x) \leq  C_q\overbar{M}e^{C_q\overbar{M}^at}, \quad \mbox{and} \quad   T(t,x) \geq C \overbar{M}^{-\frac{2}{3}}e^{-\frac{2}{3}C_q\overbar{M}^at}.
\end{align*}\unhide
\begin{align}\label{macrobarM123}
\begin{split}
&(1)~  C_0e^{-C_q\overbar{M}^at}\leq \rho(t,x) \leq C_q \overbar{M},  \\
&(2)~ |U(t,x)| \leq C_q\overbar{M}e^{C_q\overbar{M}^at},   \\
&(3)~  C \overbar{M}^{-\frac{2}{3}}e^{-\frac{2}{3}C_q\overbar{M}^at} \leq T(t,x) \leq  C_q\overbar{M}e^{C_q\overbar{M}^at},
\end{split}
\end{align}
for a generic constant $C_q$.
\end{lemma}
\begin{proof}
Note that the collision frequency $\nu(t,x)=\rho^a T^b$ for $a\geq b$ is bounded as
\begin{align}\label{nuesti}
\nu(t,x)=\rho^a T^b=(\rho)^{a-b}(\rho T)^b \leq C_{q} \sup_{0\leq s \leq t_*}\|F(s)\|_{L^{\infty,q}_{x,v}}^a,
\end{align}
by Lemma \ref{rho,T esti}. The  estimates for the macroscopic fields in Lemma \ref{rho,T esti} and the estimate for collision frequency in \eqref{nuesti} yields
\begin{align}\label{barMbdd}
(\rho, \rho U, 3\rho T+\rho|U|^2) \leq C_q \overbar{M} , \qquad \nu \leq C_q \overbar{M}^a.
\end{align}
(1) The lower bound of $\rho$ comes from the mild formulation of the BGK model,
\begin{align*}
F(t,x,v) &= e^{-\int_0^t \nu(\tau, x-v(t-\tau)) d\tau}F_0(x-vt,v)\cr
&\quad + \int_0^t e^{-\int_s^t \nu(\tau, x-v(t-\tau)) d\tau} \mathcal{M}(F)(s,x-v(t-s),v) ds.
\end{align*}
Combining with the upper bound of $\nu$ in \eqref{barMbdd} and using \eqref{initial LB}, we get
\begin{align}\label{rholower}
\begin{split}
\rho(t,x)=\int_{\mathbb{R}^3}F(t,x,v)dv &\geq  \int_{\mathbb{R}^3}e^{-\int_0^t \nu(\tau,x-v\tau) d\tau} F_0(x-vt,v) dv \geq  C_0 e^{-C_q\overbar{M}^at}.
\end{split}
\end{align}
(2) Applying \eqref{rholower}, we have
\begin{align*}
|U| = \frac{|\rho U|}{\rho} \leq \frac{C_q\overbar{M}}{C_0e^{-C_q\overbar{M}^at} } 
\leq C_q\overbar{M}e^{C_q\overbar{M}^at}.
\end{align*}
(3) Similar to (2), we have the following upper bound of the temperature.
\begin{align*}
|T| = \frac{3\rho T +\rho |U|^2}{3\rho} -\frac{1}{3}|U|^2  \leq \frac{C_q\overbar{M}}{3C_0e^{-C_q\overbar{M}^at} } 
\leq C_q\overbar{M}e^{C_q\overbar{M}^at} .
\end{align*}
For the lower bound of the temperature, we use (1) in Lemma \ref{Pesti}
and \eqref{rholower} to obtain
\begin{align*}
T^{\frac{3}{2}} \geq \frac{\rho}{C\|F\|_{L^{\infty}_{x,v}}} \geq \frac{C_0e^{-C_q\overbar{M}^at} }{C\overbar{M}} 
\geq \frac{C}{\overbar{M}}e^{-C_q\overbar{M}^at}.
\end{align*}
\end{proof}



\begin{lemma}\label{thetaesti} Let \eqref{initial LB} and the a priori assumption \eqref{Assumption} hold. Then the transitions of the macroscopic fields $(\rho_{\theta}$, $U_{\theta}$, $T_{\theta})$ enjoy the following estimates
\begin{align}\label{thetaesti123}
\begin{split}
&(1)~  C_0e^{-C_q\overbar{M}^at} \leq \rho_{\theta} \leq C_q \overbar{M}, \\
&(2)~  |U_{\theta}| \leq C_q \overbar{M}e^{C_q\overbar{M}^at},  \\
&(3)~ C_q\overbar{M}^{-\frac{5}{3}}e^{-\frac{5}{3}C_q\overbar{M}^at} \leq T_{\theta}  \leq C_q \overbar{M}e^{C_q\overbar{M}^at},
\end{split}
\end{align}
for $0\leq \theta \leq 1$ and some generic constant $C_q$.
\end{lemma}
\begin{proof}
Recall the definition of the transition of the macroscopic fields in \eqref{transition}:
\begin{align*}
\rho_{\theta} &= \theta \rho +(1-\theta), \quad \rho_{\theta}U_{\theta} = \theta \rho U, \quad 3\rho_{\theta} T_{\theta}+\rho_{\theta} |U_{\theta}|^2 - 3\rho_{\theta} = \theta (3\rho T + \rho |U|^2 - 3\rho).
\end{align*}
(1) Applying the upper and lower bound of $\rho$ from $\eqref{macrobarM123}_1$ in Lemma \ref{macrobarM}, we have
\begin{align*}
\rho_{\theta} &= \theta \rho +(1-\theta) \leq  \theta C_q \overbar{M}+ (1-\theta) \leq C_q \overbar{M}, \cr
\rho_{\theta} &= \theta \rho +(1-\theta) \geq \theta C_0e^{-C_q\overbar{M}^at}+ (1-\theta) \geq C_0e^{-C_q\overbar{M}^at}.
\end{align*}
(2) Upper bound of $\rho U$ in \eqref{barMbdd} and the lower bound of $\rho_{\theta}$ $\eqref{thetaesti123}_1$ yield
\begin{align*}
|U_{\theta}| = \bigg| \frac{\theta \rho U}{\rho_{\theta}} \bigg| \leq \frac{\theta C_q \overbar{M}}{C_0e^{-C_q\overbar{M}^at}} \leq C_q \overbar{M}e^{C_q\overbar{M}^at} .
\end{align*}
(3) By the definition of $T_{\theta}$, we have
\begin{align}\label{Ttheta}
\begin{split}
3\rho_{\theta} T_{\theta}  &= \theta (3\rho T + \rho |U|^2 - 3\rho)-\rho_{\theta} |U_{\theta}|^2 + 3\rho_{\theta} \cr
&= (3\rho T+\rho |U|^2)\theta +3(1-\theta) -\rho_{\theta} |U_{\theta}|^2,
\end{split}
\end{align}
where we used
\begin{align*}
3\rho_{\theta} - 3\theta\rho =  3(\theta\rho+(1-\theta)) - 3\theta\rho = 3(1-\theta).
\end{align*}
For an upper bound of $T_{\theta}$, we apply $\eqref{thetaesti123}_1$ and upper bound of $3\rho T+\rho|U|^2$ in \eqref{barMbdd} to obtain
\begin{align*}
T_{\theta}  &= \frac{(3\rho T+\rho |U|^2)\theta +3(1-\theta) -\rho_{\theta} |U_{\theta}|^2}{3\rho_{\theta}} \leq \frac{\theta C_q \overbar{M} +3(1-\theta)}{3C_0e^{-C_q\overbar{M}^at}} \leq C_q \overbar{M}e^{C_q\overbar{M}^at}.
\end{align*}
For a lower bound of $T_{\theta}$, we substitute the following computation
\begin{align*}
\theta \rho |U|^2 -\rho_{\theta} |U_{\theta}|^2 &= \frac{\theta \rho_{\theta}\rho |U|^2 -|\rho_{\theta} U_{\theta}|^2}{\rho_{\theta}} =\frac{\theta (\theta\rho+(1-\theta))\rho |U|^2 -\theta^2|\rho U|^2}{\rho_{\theta}} =\frac{\theta(1-\theta)\rho |U|^2}{\rho_{\theta}}
\end{align*}
into \eqref{Ttheta} to obtain
\begin{align*}
3\rho_{\theta} T_{\theta} &= (3\rho T)\theta +3(1-\theta) + \frac{\theta(1-\theta)\rho |U|^2}{\rho_{\theta}}.
\end{align*}
Then we use $\eqref{thetaesti123}_1$, $\eqref{macrobarM123}_1$, and $\eqref{macrobarM123}_3$ to get
\begin{align*}
T_{\theta}  &\geq \frac{\theta \rho T +(1-\theta)}{\rho_{\theta} } \geq \frac{(C_0e^{-C_q\overbar{M}^at})(C \overbar{M}^{-\frac{2}{3}}e^{-\frac{2}{3}C_q\overbar{M}^at})}{C_q \overbar{M}} \geq C_q\overbar{M}^{-\frac{5}{3}}e^{-\frac{5}{3}C_q\overbar{M}^at}.
\end{align*}
\end{proof}

Now we are ready to estimate the nonlinear term on the R.H.S of \eqref{Rf est 2}.
\begin{lemma}\label{nonlinMbar}
	Let \eqref{initial LB} and the a priori assumption \eqref{Assumption} hold. We have the following estimate for the third term of the right-hand side of \eqref{Rf est 2}
	\begin{align*}
	&\int_{0}^{t} e^{-(t-s)} \int_{\vert v \vert \leq N} \langle v\rangle^2\vert \Gamma(f) (s, x - v(t-s),v)\vert dvds \cr
	&\leq C_q \overbar{M}^n e^{C_q\overbar{M}^at}\bigg[ N^5(1-e^{-\lambda}) \overbar{M}+\frac{\overbar{M}}{N^{q-10}}
	+ N^6 (\lambda^{-2}+N^3) \left(\mathcal{E}(F_0)+N^{\frac{3}{2}}\sqrt{\mathcal{E}(F_0)} \right)\bigg],
	\end{align*}
	for some generic constants $n>1$ and $C_q>0$.
\end{lemma}
\begin{proof}
In this proof, we claim the following estimate:
\begin{align} \label{claim}
\langle v \rangle^{2}\vert \Gamma(f)(t,x,v)\vert \leq  C_q \overbar{M}^ne^{C_q\overbar{M}^at}\int_{\mathbb{R}^3} \vert f(u) \vert (1+|u|^2) du,
\end{align}
for some generic constants $n>1$ and $C_q>0$. \\
Recall definition of the full nonlinear term in Lemma \ref{linearize}. We first consider the nonlinear term $\Gamma_2(f)$ which contains some second derivative terms of the local Maxwellian. Note that the second derivative of the local Maxwellian can be written by the following polynomial form \eqref{polyform}:
\begin{align*}
\left[\nabla_{(\rho_{\theta},\rho_{\theta} U_{\theta}, G_{\theta})}^2 \mathcal{M}(\theta)\right]_{ij} = \frac{\mathcal{P}_{ij}((v-U_{\theta}),U_{\theta},T_{\theta})}{\rho_{\theta}^{\alpha_{ij}}T_{\theta}^{\beta_{ij}}}\mathcal{M}(\theta),
\end{align*}
where $\mathcal{P}(x_1,\cdots,x_n)$ is generic polynomial for $x_1,\cdots ,x_n$. Applying $\langle v \rangle^{2} \leq 1+ |v-U|^2 + |U|^2 $, we get
\begin{align}\label{nonlinsubsti}
\begin{split}
\bigg|\langle v \rangle^{2}\left[\nabla_{(\rho_{\theta},\rho_{\theta} U_{\theta}, G_{\theta})}^2 \mathcal{M}(\theta)\right]_{ij}\bigg| &\leq C\bigg|(1+ |v-U_{\theta}|^2 + |U_{\theta}|^2)\frac{\mathcal{P}_{ij}((v-U_{\theta}),U_{\theta},T_{\theta})}{\rho_{\theta}^{\alpha_{ij}}T_{\theta}^{\beta_{ij}}}\mathcal{M}(\theta)\bigg| \cr
&\leq C \bigg|(1+ T_{\theta} + |U_{\theta}|^2) \frac{\mathcal{P}_{ij}(\sqrt{T_{\theta}},U_{\theta},T_{\theta})}{\rho_{\theta}^{\alpha_{ij}}T_{\theta}^{\beta_{ij}}}\bigg|,
\end{split}
\end{align}
where we used the following inequality
\begin{align*}
\bigg|\frac{(v-U)^n}{T^{\frac{n}{2}}}\exp\left(-\frac{|v-U|^2}{2T}\right)\bigg| \leq C,
\end{align*}
to control the $(v-U)$ part on the numerator part. Then, putting $(1+ T_{\theta} + |U_{\theta}|^2)$ term on the generic polynomial $\mathcal{P}(\sqrt{T_{\theta}},U_{\theta},T_{\theta})$, we apply Lemma \ref{thetaesti} to estimate the transition of the macroscopic fields ($\rho_{\theta},U_{\theta},T_{\theta}$):
\begin{align*}
\bigg|\langle v \rangle^{2}\left[\nabla_{(\rho_{\theta},\rho_{\theta} U_{\theta}, G_{\theta})}^2 \mathcal{M}(\theta)\right]_{ij}\bigg|&\leq C \bigg|\frac{\mathcal{P}_{ij}(\sqrt{C_q \overbar{M}e^{C_q\overbar{M}^at}},C_q \overbar{M}e^{C_q\overbar{M}^at},C_q \overbar{M}e^{C_q\overbar{M}^at})}{(C_0e^{-C_q\overbar{M}^at})^{\alpha_{ij}}(C_q\overbar{M}^{-\frac{5}{3}}e^{-\frac{5}{3}C_q\overbar{M}^at})^{\beta_{ij}}}\bigg|.
\end{align*}
Thus, using Appendix \ref{Nonlincomp}, there exists a positive constant $n$ such that
\begin{align}\label{nonlinupper}
\bigg|\left[\nabla_{(\rho_{\theta},\rho_{\theta} U_{\theta}, G_{\theta})}^2 \mathcal{M}(\theta)\right]_{ij}\bigg|
&\leq C_q \overbar{M}^ne^{C_q\overbar{M}^at}.
\end{align}
From now on, we use the positive number $n$ as a generic positive constant. By using the estimate of the collision frequency $\nu=\rho^aT^b\leq C_q\overbar{M}^a$ in \eqref{nuesti} and
\begin{align*}
\langle f,e_i \rangle_{v} \leq C_q \| f(t)\|_{L^{\infty,q}_{x,v}} \leq C_q\overbar{M},
\end{align*}
we can bound the nonlinear term $\Gamma_2$ as
\begin{align} \label{gamma2 est}
 \Gamma_2(f) 
\leq  C_q \overbar{M}^ne^{C_q\overbar{M}^at}\int_{\mathbb{R}^3} \vert f(u)\vert (1+|u|^2) du .
\end{align}
Similarly, applying the estimates for $(\rho_{\theta},U_{\theta},T_{\theta})$ in Lemma \ref{thetaesti}, the nonlinear part of the collision frequency $A_i(\theta)$ in Lemma \ref{linearize} can be bounded as
\begin{align*}
A_i(\theta)&\leq C_q \overbar{M}^ne^{C_q\overbar{M}^at}, \quad \mbox{for} \quad i=1,\cdots,5.
\end{align*}
Combining with the estimate of $\mathbf{P}f-f$,
\begin{align*}
\|(\mathbf{P}f-f)(t)\|_{L^{\infty,q}_{x,v}} \leq  \| f(t)\|_{L^{\infty,q}_{x,v}} \leq \overbar{M},
\end{align*}
we also have
\begin{align} \label{gamma1 est}
\Gamma_1(f)  
\leq  C_q \overbar{M}^ne^{C_q\overbar{M}^at}\int_{\mathbb{R}^3} \vert f(u) \vert (1+|u|^2) du .
\end{align}
From \eqref{gamma1 est} and \eqref{gamma2 est}, we obtain \eqref{claim}. Now, applying \eqref{claim} yields
\begin{align*}
&\int_{0}^{t} e^{-(t-s)} \int_{\vert v \vert \leq N} \langle v\rangle^2\vert \Gamma(f) (s, x - v(t-s),v)\vert dvds \cr
&\leq C_q \overbar{M}^n e^{C_q\overbar{M}^at} \int_{0}^{t} e^{-(t-s)}  \int_{\vert v \vert \leq N} \langle v\rangle^2 \int_{\mathbb{R}^3} \vert f(s,x-v(t-s),u)\vert (1+|u|^2) du  dvds.
\end{align*}
We note that the R.H.S is the same with the estimate of $\mathbf{P}f$ in Lemma \ref{Pfesti} except the term $C_q \overbar{M}^n e^{C_q\overbar{M}^at}$. Thus we finish the proof.
\end{proof}

\hide
\begin{align*}
3-\delta -\rho|U|^2  \leq 3\rho T\leq 3+\delta -\rho|U|^2
\end{align*}
\begin{align*}
\frac{3-\delta -\delta(1+\delta)}{3\rho}  \leq  T\leq \frac{3+\delta}{3\rho}
\end{align*}
\unhide

Now we go back to the proof of Proposition \ref{macrodecay}.
\begin{proof} [Proof of Proposition \ref{macrodecay}]
Combining Lemma \ref{termdivide}, Lemma \ref{Pfesti} and Lemma \ref{nonlinMbar}, we obtain
\begin{align*}
\int_{\mathbb{R}^3} \langle v\rangle^2\vert f(t,x,v)\vert dv &\leq C_{q}M_0e^{-t}+ \frac{C_q}{N^{q-5}} \overbar{M}  + C_q \overbar{M}^n e^{C_q\overbar{M}^at}\bigg[ N^5(1-e^{-\lambda}) \overbar{M}+\frac{\overbar{M}}{N^{q-10}} \cr
&\quad + N^6 (\lambda^{-2}+N^3) \left(\mathcal{E}(F_0)+N^{\frac{3}{2}}\sqrt{\mathcal{E}(F_0)} \right)\bigg].
\end{align*}
By using the generic constants $C_q$ and $n$, for the time $t_*$ which satisfies \eqref{Assumption}, we can write the above inequality as
\begin{align}\label{makesmall}
\begin{split}
\int_{\mathbb{R}^3} \langle v\rangle^2\vert f(t,x,v) \vert dv &\leq C_{q}M_0e^{-t}+ \frac{C_q}{N^{q-5}} \overbar{M} +C_q \overbar{M}^n e^{C_q\overbar{M}^at_*} \frac{1}{N^{q-10}} \cr
&\quad + C_q \overbar{M}^n e^{C_q\overbar{M}^at_*}N^5(1-e^{-\lambda}) \cr
&\quad + C_q \overbar{M}^n e^{C_q\overbar{M}^at_*}N^6 (\lambda^{-2}+N^3) \left(\mathcal{E}(F_0)+N^{\frac{3}{2}}\sqrt{\mathcal{E}(F_0)} \right).
\end{split}
\end{align}
Then we choose $N,\lambda,\mathcal{E}(F_0)$ to make the second to the fifth terms of the R.H.S in \eqref{makesmall} sufficiently small.
First, for a given $\delta\in(0,1)$, let us choose a constant $N$ sufficiently large as follows:
\begin{align} \label{choose N}
N:= \max\left\{ \left(\frac{8}{\delta}C_q\overbar{M}\right)^{\frac{1}{q-5}}, \left(\frac{8}{\delta}C_q \overbar{M}^n e^{C_q\overbar{M}^at_*}\right)^{\frac{1}{q-10}} \right\},\quad q > 10.
\end{align}
Then the second and third terms on the R.H.S of \eqref{makesmall} become smaller than $\delta/4$.
Now, for $N$ which was chosen in \eqref{choose N}, we choose sufficiently small $\lambda $ as
\begin{align*}
\lambda := -\ln \left(1-\frac{\delta}{4C_q \overbar{M}^n e^{C_q\overbar{M}^aT}N^5}\right),
\end{align*}
to make the second line of the R.H.S of \eqref{makesmall} small:
\begin{align*}
C_q \overbar{M}^n e^{C_q\overbar{M}^at_*}N^5(1-e^{-\lambda}) = \frac{\delta}{4}.
\end{align*}
Finally, choosing sufficiently small initial entropy satisfying
\begin{align*}
\mathcal{E}(F_0) \leq \min\left\{\frac{\delta}{8C_q \overbar{M}^n e^{C_q\overbar{M}^at_*}N^6 (\lambda^{-2}+N^3)}, \left(\frac{\delta}{8C_q \overbar{M}^n e^{C_q\overbar{M}^at_*}N^{\frac{15}{2}} (\lambda^{-2}+N^3)}\right)^2 \right\},
\end{align*}
the third line on the R.H.S of \eqref{makesmall} become smaller than $\delta/4$:
\begin{align*}
C_q \overbar{M}^n e^{C_q\overbar{M}^at_*}N^6 (\lambda^{-2}+N^3) \left(\mathcal{E}(F_0)+N^{\frac{3}{2}}\sqrt{\mathcal{E}(F_0)} \right) \leq \frac{\delta}{4},
\end{align*}
which implies
\Be \notag
\left \vert \int_{\mathbb{R}^3} \langle v\rangle^2f(t,x,v) dv  \right \vert \leq C_qM_0e^{-t}  + \frac{3}{4}\delta.
\Ee
This completes the proof of Proposition \ref{macrodecay}.
\end{proof}

\begin{lemma}\label{macrodelta} Under a priori assumption \eqref{Assumption} with sufficiently small $\mathcal{E}(F_0$) satisfying the assumption of Proposition \ref{macrodecay}, there exists $t_{eq}$ such that if $t\geq t_{eq}$ then the macroscopic fields are close to the global macroscopic fields for any $\delta\in(0,1/3)$:
\begin{align*}
|\rho-1|~, |U|,~ |T-1| \leq 2\delta.
\end{align*}
\end{lemma}
\begin{proof}
From Proposition \ref{macrodecay}, let us choose sufficiently large time
\begin{align*}
t\geq \ln \frac{4C_qM_0}{\delta},
\end{align*}
for which the followings hold:
\Be \notag
\left \vert \int_{\mathbb{R}^3} \langle v\rangle^2f(t,x,v) dv  \right \vert \leq \delta.
\Ee
This is equivalent to
\begin{align*}
|\rho-1|,~|\rho U|,~|3\rho T+\rho|U|^2-3| \leq \delta.
\end{align*}
Then the macroscopic velocity and temperature are bounded by
\begin{align*}
|U(t,x)| &= \frac{|\rho U|}{\rho} \leq \frac{\delta}{1-\delta},\quad
T(t,x) \leq \frac{3+\delta}{3(1-\delta)}, \quad
T(t,x) \geq \frac{3-\delta}{3(1+\delta)} - \frac{\delta^2}{3(1-\delta)^2} ,
\end{align*}
where we used the relation $T=\frac{3\rho T+\rho|U|^2}{3\rho}-\frac{|\rho U|^2}{3\rho^2}$. Once we consider the quantity $|T-1|$, then we have
\begin{align*}
|U| \leq \frac{\delta}{1-\delta}, \qquad |T-1| \leq  \max\left\{\frac{4\delta}{3(1-\delta)},\frac{4\delta}{3(1+\delta)}+ \frac{\delta^2}{3(1-\delta)^2} \right\} .
\end{align*}
Thus for $\delta \leq 1/3$, we have
\begin{align*}
|\rho-1|,~|U|,~|T-1| \leq 2\delta.
\end{align*}
\end{proof}
For later convenience, we define the time satisfying Lemma \ref{macrodelta} as
\begin{align} \label{def teq}
t_{eq}:= \ln \frac{4C_qM_0}{\delta}.
\end{align}
Note that $t_{eq}$ depends only on the initial data, and fixed constants $q>10$ and $\delta$. After this time, the nonlinear term $\Gamma$ becomes quadratic nonlinear. We will consider the problem after $t_{eq}$ in Section \ref{secsmall}. Now, our main problem is to construct the solution before $t_{eq}$.

\section{Local existence theory}
In this section, we consider the local-in-time unique solution of the BGK equation.
\begin{lemma}\label{nonlinfg} If the two distribution functions $F=\mu+f$ and $G=\mu+g$ satisfy $\|F(t)\|_{L^{\infty,q}_{x,v}} \leq M$, and $\|G(t)\|_{L^{\infty,q}_{x,v}} \leq M$, for a constant $M>0$, and the macroscopic fields of $F$ and $G$ satisfy Lemma \ref{macrobarM} for $M$ instead of $\overbar{M}$, respectively, then we have
	\begin{align*}
	\| (\Gamma(f)-\Gamma(g))(t) \|_{L^{\infty,q}_{x,v}} \leq C_{M}\|(f-g)(t)\|_{L^{\infty,q}_{x,v}},
	\end{align*}
	for a positive constant $C_{M}$.
\end{lemma}
\begin{proof}
We denote the macroscopic fields of $F=\mu+f$ as $(\rho^f,U^f,T^f)$:
\begin{align*}
&\rho^f(t,x):= \int_{\mathbb{R}^3} F(t,x,v) dv, \\
&\rho^f(t,x)U^f(t,x):= \int_{\mathbb{R}^3} F(t,x,v) v dv, \\
&3\rho^f(t,x)T^f(t,x) := \int_{\mathbb{R}^3} F(t,x,v) \vert v - U^f(t,x) \vert^2 dv.
\end{align*}
We also write the transition of the macroscopic fields comes from the definition \eqref{transition} as $(\rho_{\theta}^f,U_{\theta}^f,T_{\theta}^f)$ and the local Maxwellian depending on the macroscopic fields $(\rho_{\theta}^f,U_{\theta}^f,T_{\theta}^f)$ as $\mathcal{M}^f(\theta)$.
Here, we only consider the case $\Gamma_2(f)$, since $\Gamma_{1}(f)$ can be treated similarly.
We split the function dependency as follows:
	\begin{align*}
	\Gamma_2(f_1,f_2,f_3,f_4) &= (\rho^{f_1})^a (T^{f_1})^b\sum_{1\leq i,j\leq 5}\int_0^1 \frac{\mathcal{P}_{ij}((v-U_{\theta}^{f_2}),U_{\theta}^{f_2},T_{\theta}^{f_2})}{(\rho_{\theta}^{f_2})^{\alpha_{ij}}(T_{\theta}^{f_2})^{\beta_{ij}}}\mathcal{M}^{f_2}(\theta)  (1-\theta) d\theta \cr
	&\quad \times \int_{\mathbb{R}^3}f_3 e_i dv \int_{\mathbb{R}^3}f_4 e_jdv.
	\end{align*}
	By the triangle inequality, we have
	\begin{align*}
	\langle v \rangle^q(\Gamma_2(f)-\Gamma_2(g)) &=\langle v \rangle^q(\Gamma_2(f,f,f,f)-\Gamma_2(f,f,f,g)) +\langle v \rangle^q(\Gamma_2(f,f,f,g)-\Gamma_2(f,f,g,g)) \cr
	& \quad +\langle v \rangle^q(\Gamma_2(f,f,g,g)-\Gamma_2(f,g,g,g)) +\langle v \rangle^q(\Gamma_2(f,g,g,g)-\Gamma_2(g,g,g,g)) \cr
	&:= I+II+III+IV.
	\end{align*}
	Applying the same estimate in \eqref{nonlinupper} with
	\begin{align*}
	\langle v \rangle^q \mathcal{M}^f(\theta) \leq \left(1+|v-U_{\theta}^f|^q+ |U_{\theta}^f|^q\right)\mathcal{M}^f(\theta),
	\end{align*}
	we obtain
	\begin{align*}
	I+II \leq C_{M} \int_{\mathbb{R}^3} (1+|v|+|v|^2)|f-g|dv \leq C_{M}\|f-g\|_{L^{\infty,q}_{x,v}},
	\end{align*}
because the macroscopic fields $(\rho,U,T)$ and $(\rho_{\theta},U_{\theta},T_{\theta})$ for $f$ and $g$ are bounded with respect to $M$.
Before we estimate $III$ and $IV$, we consider the following two inequalities: For any real numbers $\alpha,\beta \in \mathbb{R}$,
by the mean value theorem, there exists $\rho_c$ with $\min\{\rho^f,\rho^g\} \leq \rho_{c} \leq \max\{\rho^f,\rho^g\}$ such that
\begin{align}\label{rhopower}
(\rho^f)^{\alpha} -(\rho^g)^{\alpha} = \alpha\rho_c^{\alpha-1}(\rho^f-\rho^g)\leq C_{M}\int_{\mathbb{R}^3}|f-g|dv \leq C_{M}\|f-g\|_{L^{\infty,q}_{x,v}}.
\end{align}
Similarly, there exists $(\rho T)_c$ with $\min\{\rho^fT^f,\rho^gT^g\} \leq (\rho T)_c \leq \max\{\rho^fT^f,\rho^gT^g\}$ such that
\begin{align}\label{rhoTpower}
(\rho^fT^f)^{\beta}-(\rho^gT^g)^{\beta} = \beta(\rho_cT_c)^{\beta-1}(\rho^fT^f-\rho^gT^g) \leq C_{M}|\rho^fT^f-\rho^gT^g|.
\end{align}
Now we consider the fourth term $IV$. Similar to the estimate of $I$ and $II$, we have
\begin{align*}
IV \leq C_{M}((\rho^f)^a (T^f)^b-(\rho^g)^a (T^g)^b).
\end{align*}
We apply the triangle inequality and split the terms as follows:
\begin{align*}
IV 
&\leq C_{M}\left\{ ((\rho^f)^{a-b} -(\rho^g)^{a-b}) (\rho^gT^g)^b+ (\rho^f)^{a-b} ((\rho^fT^f)^b-(\rho^gT^g)^b)\right\} \cr
&= IV_1 + IV_2.
\end{align*}
Then \eqref{rhopower} and \eqref{rhoTpower} guarantee $IV_1\leq C_{M}\|f-g\|_{L^{\infty,q}_{x,v}}$ and $IV_2 \leq C_{M}|\rho^fT^f-\rho^gT^g|$, respectively. To make $\int|v|^2f dv$ and $\int|v|^2g dv$ terms, we consider
	\begin{align*}
	\rho^fT^f-\rho^gT^g &= \left(\rho^fT^f+\frac{1}{3}\rho^f|U^f|^2-\rho^gT^g-\frac{1}{3}\rho^g|U^g|^2\right) + \frac{1}{3}\left(\rho^g|U^g|^2-\rho^f|U^f|^2\right) \cr
	&= \frac{1}{3}\int_{\mathbb{R}^3} (f-g)|v|^2 + \frac{1}{3}\left(\rho^g|U^g|^2-\rho^f|U^f|^2\right).
	\end{align*}
	For the last quantity, we apply
	\begin{align*}
	\rho^g|U^g|^2-\rho^f|U^f|^2 &= \frac{|\rho^gU^g|^2}{\rho^g} -\frac{|\rho^fU^f|^2}{\rho^f}  \cr
	&= \frac{|\rho^gU^g|^2}{\rho^g}-\frac{|\rho^fU^f|^2}{\rho^g}+\frac{|\rho^fU^f|^2}{\rho^g} -\frac{|\rho^fU^f|^2}{\rho^f} \cr
	&\leq C_{M}\int_{\mathbb{R}^3}|v||f-g|dv +  C_{M}\int_{\mathbb{R}^3}|f-g|dv \cr
	&\leq C_{M}\|f-g\|_{L^{\infty,q}_{x,v}},
	\end{align*}
	to have
	\begin{align*}
	|\rho^fT^f-\rho^gT^g| \leq C_{M}\|f-g\|_{L^{\infty,q}_{x,v}}.
	\end{align*}
	Thus we obtain $IV_2 \leq C_{M}\|f-g\|_{L^{\infty,q}_{x,v}}$. Finally, for $III$ term, we have
	\begin{align*}
	III &\leq C_{M} \Bigg| \sum_{1\leq i,j\leq 5} \int_0^1 \langle v \rangle^q\left(\frac{\mathcal{P}_{ij}((v-U_{\theta}^f),U_{\theta}^f,T_{\theta}^f)}{(\rho_{\theta}^f)^{\alpha_{ij}}(T_{\theta}^f)^{\beta_{ij}}}\mathcal{M}^f(\theta)-\frac{\mathcal{P}_{ij}((v-U_{\theta}^g),U_{\theta}^g,T_{\theta}^g)}{(\rho_{\theta}^g)^{\alpha_{ij}}(T_{\theta}^g)^{\beta_{ij}}}\mathcal{M}^g(\theta)\right)   (1-\theta) d\theta \Bigg|.
	\end{align*}
	We split the terms inside $III$ as follows:
	\begin{align*}
	III_1^{ij} &= \left(\frac{\mathcal{P}_{ij}((v-U_{\theta}^f),U_{\theta}^f,T_{\theta}^f)}{(\rho_{\theta}^f)^{\alpha_{ij}}(T_{\theta}^f)^{\beta_{ij}}}-\frac{\mathcal{P}_{ij}((v-U_{\theta}^g),U_{\theta}^g,T_{\theta}^g)}{(\rho_{\theta}^g)^{\alpha_{ij}}(T_{\theta}^g)^{\beta_{ij}}}\right)\langle v \rangle^q \mathcal{M}^f(\theta), \cr
	III_2^{ij} &= \frac{\mathcal{P}_{ij}((v-U_{\theta}^g),U_{\theta}^g,T_{\theta}^g)}{(\rho_{\theta}^g)^{\alpha_{ij}}(T_{\theta}^g)^{\beta_{ij}}}\langle v \rangle^q (\mathcal{M}^f(\theta)-\mathcal{M}^g(\theta)).
	\end{align*}
	For the $III_1^{ij}$ term, we apply triangle inequality several times and use the estimate \eqref{rhopower} and \eqref{rhoTpower} to have $III_1^{ij} \leq C_{M}\|f-g\|_{L^{\infty,q}_{x,v}}$. For the term $\langle v \rangle^q (\mathcal{M}^f(\theta)-\mathcal{M}^g(\theta))$ inside $III_2^{ij}$, we can have the Lipschitz continuity of the local Maxwellian $III_2^{ij}\leq C_{M}\|f-g\|_{L^{\infty,q}_{x,v}}$ as in \cite{Perthame,Yun2}.
This completes the proof.
\end{proof}

We prove the local wellposedness theory of the BGK solutions.
\begin{proposition}\label{local} Consider the BGK equation \eqref{reform BGK} with initial data $f_0$ which satisfies $\|f_0\|_{L^{\infty,q}_{x,v}} <\infty$ and \eqref{initial LB}. Then, there exists a time $t_0=t_0(\Vert f_0 \Vert_{L^{\infty,q}_{x,v}})$ depending only on $\Vert f_0 \Vert_{L^{\infty,q}_{x,v}}$ such that there exists a unique local-in-time solution for $t\in [0,t_0]$ which safisfies
\begin{align*}
\sup_{0\leq s\leq t_0}\|f(s)\|_{L^{\infty,q}_{x,v}}\leq 2\Vert f_0 \Vert_{L^{\infty,q}_{x,v}}.
\end{align*}
\end{proposition}
\begin{proof}
	From \eqref{reform BGK} in Lemma \ref{linearize}, we obtain the following mild solution:
	\begin{align*}
	f(t,x,v) = e^{-t} f_0(x-vt,v) + \int_0^t e^{-(t-s)}[\mathbf{P}f+\Gamma(f)](s,x-v(t-s),v) ds.
	\end{align*}
	Multiplying the above by $\langle v \rangle^q$, we have
	\begin{align*}
	\langle v \rangle^q \vert f(t,x,v)\vert \leq e^{-t} \Vert f_0 \Vert_{L^{\infty,q}_{x,v}} + \int_0^t e^{-(t-s)} \langle v \rangle^q \vert[ \mathbf{P}f + \Gamma (f)](s,x-v(t-s),v)\vert ds.
	\end{align*}
	By definition \eqref{pf} of $\mathbf{P}f$ in Lemma \ref{linearize}, we directly deduce that
	\begin{align}\label{PfLq}
	\langle v \rangle^q \mathbf{P}f(s) \leq C_q\sup_{0\leq s \leq t} \Vert f(s) \Vert_{L^{\infty,q}_{x,v}}, \quad 0\leq s \leq t.
	\end{align}
	To obtain estimate for $\Gamma(f)$, we note that the macroscopic fields $(\rho,U,T)$ and $(\rho_{\theta},U_{\theta},T_{\theta})$ are bounded depending on $M_0$ by the assumption \eqref{initial LB} with Lemma \ref{macrobarM} and Lemma \ref{thetaesti}. Then using a similar argument in the proof of \eqref{claim}, we have
	\begin{align*}
	\langle v \rangle^q \vert \Gamma(f)(s)\vert \leq C_q \left(\sup_{0\leq s \leq t}\Vert f(s) \Vert_{L^{\infty,q}_{x,v}}\right)^n e^{C_q (\sup_{0\leq s \leq t}\Vert f(s) \Vert_{L^{\infty,q}_{x,v}})^at}, \quad 0\leq s \leq t,
	\end{align*}
	where $C_q,n$ and $a$ are the same constant as in \eqref{claim}. Note that \eqref{claim}  holds for $\sup_{0\leq s \leq t}  \Vert f(s) \Vert_{L^{\infty,q}_{x,v}}$ instead of $\overbar{M}$. In sum, we obtain
	\begin{align*}
	\Vert f(t) \Vert_{L^{\infty,q}_{x,v}} &\leq e^{-t} \Vert f_0 \Vert_{L^{\infty,q}_{x,v}}+C_q(1-e^{-t})\sup_{0\leq s \leq t} \Vert f(s) \Vert_{L^{\infty,q}_{x,v}} \cr
	&\quad + C_q \left(\sup_{0\leq s \leq t}\Vert f(s) \Vert_{L^{\infty,q}_{x,v}}\right)^{n} te^{C_q (\sup_{0\leq s \leq t}\Vert f(s) \Vert_{L^{\infty,q}_{x,v}})^at}.
	\end{align*}
	Hence, there exists $t_0=t_0(M_0)$ such that
	\begin{align*}
	\sup_{0\leq t \leq t_0} \Vert f(s) \Vert_{L^{\infty,q}_{x,v}} \leq 2\Vert f_0 \Vert_{L^{\infty,q}_{x,v}}.
	\end{align*}
	For the uniqueness of solutions, we assume that $g$ satisfies the reformulated BGK equation \eqref{reform BGK} with the same initial data $f_0$ and
	\begin{align*}
	\sup_{0\leq s \leq t_0} \Vert g(s) \Vert_{L^{\infty,q}_{x,v}} \leq 2\Vert f_0 \Vert_{L^{\infty,q}_{x,v}}.
	\end{align*}
	Using the mild formulation, we have
	\begin{align*}
	(f-g)(t,x,v) = \int_0^te^{-(t-s)} [\mathbf{P}(f-g) +\Gamma(f)- \Gamma(g)](s,x-v(t-s),v) ds.
	\end{align*}
	Moreover, one obtains that
	\begin{align*}
	\Vert (f-g)(t) \Vert_{L^{\infty,q}_{x,v}} \leq \int_0^t e^{-(t-s)} \left(\Vert \mathbf{P}(f-g)(s)\Vert_{L^{\infty,q}_{x,v}} + \Vert (\Gamma(f)-\Gamma(g))(s)\Vert_{L^{\infty,q}_{x,v}}\right) ds.
	\end{align*}
	For the $\mathbf{P}(f-g)$ part above, it follows from \eqref{PfLq} that
	\begin{align*}
	\int_0^t e^{-(t-s)} \Vert \mathbf{P}(f-g)(s)\Vert_{L^{\infty,q}_{x,v}}  ds \leq C_q \int_0^t  \Vert (f-g)(s) \Vert_{L^{\infty,q}_{x,v}} ds.
	\end{align*}
	To treat the nonlinear part $\Gamma(f)-\Gamma(g)$, using Lemma \ref{nonlinfg}, we obtain
	\begin{align*}
	\int_0^t e^{-(t-s)}\langle v \rangle^q \Vert (\Gamma(f)-\Gamma(g))(s)\Vert_{L^{\infty,q}_{x,v}} ds \leq C_{f_0} \int_0^t \Vert (f-g)(s)\Vert_{L^{\infty,q}_{x,v}}ds,
	\end{align*}
	where $C_{f_0}$ is a constant depending on $\Vert f_0 \Vert_{L^{\infty,q}_{x,v}}$.
	Combining the estimates for $\mathbf{P}(f-g)$ and $\Gamma(f)-\Gamma(g)$, one obtains that
	\begin{align*}
	\Vert (f-g)(t) \Vert_{L^{\infty,q}_{x,v}} \leq C_{q,f_0} \int_0^t \Vert (f-g)(s)\Vert_{L^{\infty,q}_{x,v}}ds,
	\end{align*}
	which implies the uniqueness due to Gr\"{o}nwall's inequality.
\end{proof}

\section{Control over the highly nonlinear regime and decay estimate} \label{sec 5}
In Section 3, we proved that the BGK equation enters the quadratic nonlinear regime for $t\geq t_{eq}$. In this section we consider the highly nonlinear regime $0\leq t \leq t_{eq}$ in Figure \ref{fig}. Note that we already performed estimates for macroscopic fields in Lemma \ref{macrobarM}. Unfortunately, however, the bounds of the macroscopic fields in Lemma \ref{macrobarM} depend on the a priori bound $\overbar{M}$, which is not sufficient for our estimates. In this section, instead, our aim is to control macroscopic fields $(\rho, U, T)$ for highly nonlinear regime depending only on the size of initial data $M_0$ and other generic quantities.


\subsection{Control of the macroscopic fields}

\hide
{\color{red}
We consider the control of the nonlinear term depending only on the initial data under a priori boundednes which is very similar to subsection \ref{3.nonlinear}.
\hide
{\color{red}
\begin{proposition}\label{propnonlin} Under the following a priori assumption
\begin{align*}
\sup_{0\leq s \leq t_*} \| f(s)\|_{L^{\infty,q}_{x,v}} \leq \overbar{M},
\end{align*}
we have
\begin{align*}
\sup_{0\leq s\leq t_*}|\mathcal{Q}_{i,j}(s)| \leq C_{f_0},
\end{align*}
for some large constant $C_{f_0}$ depending on initial data $f_0$.
\end{proposition}
}
\unhide

To estimate the nonlinear term, we need to estimate the macroscopic fields. Thus the estimate is very similar with Lemma \ref{macrobarM} and Lemma \ref{thetaesti}. The only different thing is that in subsection \ref{3.nonlinear}, we used the a priori bound $\sup_{0\leq s \leq t_*} \| f(s)\|_{L^{\infty,q}_{x,v}} \leq \overbar{M}$ but now we should bound it by an initial data.
In Section 3.1 we used $\nu(t,x) \leq \overbar{M}^a$ in \eqref{weightesti}.
In this section, because of the upper bound of $\rho$ and $\rho T$ in Proposition \ref{macrodecay}, we can have initial data depending upper bound $\nu(t,x)\leq \nu^{max}_{f_0}$.
} 
\unhide

\begin{lemma}\label{4.macro} Assume that the initial data satisfies \eqref{initial LB} and $\Vert f_0 \Vert_{L^{\infty,q}_{x,v}} \leq M_0<\infty$ for $q>10$. Under a priori assumption \eqref{Assumption}, the macroscopic fields $(\rho,U,T)$ in \eqref{macro quan} are bounded as follows:
\begin{align}\label{macroest123}
\begin{split}
&(1)~  C_0e^{-\nu^{max}_{f_0}t}\leq \rho(t,x) \leq C_qM_0,  \\
&(2)~ |U(t,x)| \leq  C_q M_0 e^{\nu^{max}_{f_0}t},  \\
&(3)~ CM_0^{-\frac{2}{3}}e^{-\frac{4}{3}\nu^{max}_{f_0}t}  \leq T(t,x) \leq C_q M_0 e^{\nu^{max}_{f_0}t} ,
\end{split}
\end{align}
for all $0\leq t \leq t_*$ and a generic constant $C_q>0$, where $\nu_{f_0}^{max}$ is a generic constant that will be defined in \eqref{numax}.
\end{lemma}
\begin{proof}
By Proposition \ref{macrodecay}, we have the following upper bounds for $t\in[0,t_*]$:
\begin{align} \label{macro upper}
\rho-1,~|\rho U|,~3\rho T+\rho|U|^2-3 \leq C_q M_0 + \frac{3}{4}\delta.
\end{align}
Thus the collision frequency is bounded during the time $t\in[0,t_*]$:
\begin{align*}
\nu(t,x) = \rho^{a-b} (\rho T)^b  \leq \left(C_q M_0 + 3+\frac{3}{4}\delta\right)^a,
\end{align*}
where we used $a\geq b$ and \eqref{macro upper}. We define the maximum value of the collision frequency as
\begin{align}\label{numax}
\nu^{max}_{f_0}:= \left(C_qM_0 +3+ \frac{3}{4}\delta\right)^a.
\end{align}
For simplicity, since we are considering the case $M_0>1$ with $\delta<1/3$, we write the upper bound of $\rho$, $\rho U$, and $3\rho T+\rho|U|^2$ as $C_qM_0$ by using generic constant $C_q$:
\begin{align}\label{upperrut}
\rho,~|\rho U|,~3\rho T+\rho|U|^2 \leq C_q M_0.
\end{align}
(1) By exactly the same argument as \eqref{rholower}, the mild formulation of the BGK model gives the following lower bound of the density
\begin{align*}
\rho(t,x)=\int_{\mathbb{R}^3}F(t,x,v)dv &\geq  e^{-\int_0^t \nu^{max}_{f_0} d\tau} \int_{\mathbb{R}^3}F_0(x-vt,v) dv \geq C_0e^{-\nu^{max}_{f_0}t},
\end{align*}
where the last inequality comes from \eqref{initial LB}. \\
(2) From $\eqref{macroest123}_1$ (lower bound of $\rho$) and \eqref{upperrut}, we obtain the upper bound of $U$
\begin{align*}
|U| = \frac{|\rho U|}{\rho} \leq \frac{C_q M_0}{C_0e^{-\nu^{max}_{f_0}t} }
\leq C_q M_0 e^{\nu^{max}_{f_0}t}.
\end{align*}
(3) For the upper bound of the temperature $T$, we use $\eqref{macroest123}_1$ 
and \eqref{upperrut} to obtain
\begin{align*}
|T| = \frac{3\rho T +\rho |U|^2}{3\rho} -\frac{1}{3}|U|^2  \leq \frac{C_q M_0}{3C_0e^{-\nu^{max}_{f_0}t}} \leq C_q M_0 e^{\nu^{max}_{f_0}t} .
\end{align*}
For the lower bound of the temperature $T$, we apply (1) in Lemma \ref{Pesti}. Before that, we change the $L^{\infty}$ norm of $F(t)$ to the $L^{\infty}$ norm of the initial data. By using the uniform upper bound \eqref{numax} of the collision frequency $\nu$ and (2) in Lemma \ref{Pesti}, we obtain
\begin{align*}
\avv^{q}\p_{t}F + \avv^{q}v\cdot\nabla_x F &= \nu(\avv^{q}\mathcal{M}(F) - \avv^{q}F)
\leq \nu^{max}_{f_0} \|F\|_{L^{\infty,q}_{x,v}},
\end{align*}
for $q=0$ or $q>5$.
The mild formulation gives
\begin{align*}
\|F(t)\|_{L^{\infty,q}_{x,v}} \leq \|F_0\|_{L^{\infty,q}_{x,v}} + \nu^{max}_{f_0}\int_0^t \|F(s)\|_{L^{\infty,q}_{x,v}}ds,
\end{align*}
and Gr\"{o}nwall's inequality yields
\begin{align*}
\|F(t)\|_{L^{\infty,q}_{x,v}} \leq e^{\nu^{max}_{f_0}t}\|F_0\|_{L^{\infty,q}_{x,v}}.
\end{align*}
Combining with (1) in Lemma \ref{Pesti} and $\eqref{macroest123}_1$ (lower bound of $\rho$) gives
\begin{align*}
T^{\frac{3}{2}} \geq \frac{\rho}{C\|F(t)\|_{L^{\infty}_{x,v}}} \geq \frac{C_0e^{-\nu^{max}_{f_0}t} }{Ce^{\nu^{max}_{f_0}t}\|F_0\|_{L^{\infty}_{x,v}}} \geq CM_0^{-1}e^{-2\nu^{max}_{f_0}t}.
\end{align*}
Thus, we get the lower bound of $T$ depending on the initial data.
\end{proof}

\begin{lemma}\label{4.theta} We suppose all assumptions in Lemma \ref{4.macro}. For all $0\leq t \leq t_*$, the transition of the macroscopic fields $(\rho_{\theta}$, $U_{\theta}$, $T_{\theta})$ in \eqref{transition} has the following estimate
\begin{align*}
&(1)~  C_0e^{-\nu^{max}_{f_0}t} \leq \rho_{\theta} \leq C_qM_0, \cr
&(2)~  |U_{\theta}| \leq C_q M_0e^{\nu^{max}_{f_0}t}, \cr
&(3)~ C_qM_0^{-\frac{5}{3}}e^{-\frac{7}{3}\nu^{max}_{f_0}t} \leq T_{\theta}  \leq C_q M_0e^{\nu^{max}_{f_0}t},
\end{align*}
for $0\leq \theta \leq 1$ and a generic constant $C_q$, where we have denoted $\nu_{f_0}^{max}$ as \eqref{numax}.
\end{lemma}
\begin{proof}
(1) Applying $\eqref{macroest123}_1$ (upper and lower bound of $\rho$) in Lemma \ref{4.macro}, we have
\begin{align}\label{rho theta}
\begin{split}
\rho_{\theta} &= \theta \rho +(1-\theta) \leq  \theta C_qM_0 \leq C_qM_0, \\
\rho_{\theta} &= \theta \rho +(1-\theta) \geq \theta C_0e^{-\nu^{max}_{f_0}t}+ (1-\theta) \geq C_0e^{-\nu^{max}_{f_0}t}.
\end{split}
\end{align}
(2) It follows from the upper bound of $\rho U$ in \eqref{upperrut} and $\eqref{rho theta}_2$
\begin{align*}
|U_{\theta}| = \bigg| \frac{\theta \rho U}{\rho_{\theta}} \bigg| \leq \frac{\theta C_q M_0}{C_0e^{-\nu^{max}_{f_0}t}} \leq C_q M_0e^{\nu^{max}_{f_0}t} .
\end{align*}
(3) We first derive the upper bound of $T_\theta$. By $\eqref{rho theta}_2$
 and \eqref{upperrut}, we obtain
\begin{align*}
T_{\theta}  &= \frac{(3\rho T+\rho |U|^2)\theta +3(1-\theta) -\rho_{\theta} |U_{\theta}|^2}{3\rho_{\theta}} \leq \frac{\theta C_q M_0 +3(1-\theta)}{3C_0e^{-\nu^{max}_{f_0}t}} \leq C_q M_0e^{\nu^{max}_{f_0}t}.
\end{align*}
For the lower bound of $T_\theta$, we use a similar argument in the proof of Lemma \ref{thetaesti}. Using $\eqref{rho theta}_1$
and Lemma \ref{4.macro}, we have
\begin{align*}
T_{\theta}  &\geq \frac{\theta \rho T +(1-\theta)}{\rho_{\theta} } \geq \frac{(C_0e^{-\nu^{max}_{f_0}t})(CM_0^{-\frac{2}{3}}e^{-\frac{4}{3}\nu^{max}_{f_0}t})}{C_qM_0} \geq C_qM_0^{-\frac{5}{3}}e^{-\frac{7}{3}\nu^{max}_{f_0}t}.
\end{align*}
\end{proof}

Using previous Lemma \ref{4.macro} and Lemma \ref{4.theta}, we derive an improved nonlinear estimate for $\Gamma$ than \eqref{claim}.
\begin{lemma}\label{nonlinM0} Let all assumptions in Lemma \ref{4.macro} hold. Recall the nonlinear term $\Gamma(f)$ in \eqref{def gamma}. The $v$-weighted nonlinear term is bounded as follows
\begin{align*}
\langle v \rangle^q\vert \Gamma(f)(t,x,v)\vert  \leq  C_q M_0^ne^{C\nu^{max}_{f_0}t}\Vert f(t)\Vert_{L^{\infty,q}_{x,v}}\int_{\mathbb{R}^3} \vert f(u)\vert  (1+|u|^2) du,
\end{align*}
where $n>1$ and $C_q>0$ are generic constants.
\end{lemma}
\begin{proof}
In the same way with \eqref{claim}, we substitute Lemma \ref{4.theta} for the estimate \eqref{nonlinsubsti} to get
\begin{align} \label{second derv est}
	\begin{split}
\langle v \rangle^q\bigg|\left[\nabla_{(\rho_{\theta},\rho_{\theta} U_{\theta}, G_{\theta})}^2 \mathcal{M}(\theta)\right]_{ij}\bigg| &=\langle v \rangle^q \bigg|\frac{\mathcal{P}_{ij}((v-U_{\theta}),U_{\theta},T_{\theta})}{\rho_{\theta}^{\alpha_{ij}}T_{\theta}^{\beta_{ij}}}\mathcal{M}(\theta)\bigg|\\
&\leq C\bigg|(1+ |v-U_{\theta}|^q + |U_{\theta}|^q)\frac{\mathcal{P}_{ij}((v-U_{\theta}),U_{\theta},T_{\theta})}{\rho_{\theta}^{\alpha_{ij}}T_{\theta}^{\beta_{ij}}}\mathcal{M}(\theta)\bigg|\\
&\leq C \bigg|\frac{\mathcal{P}_{ij}(\sqrt{T_{\theta}},U_{\theta},T_{\theta})}{\rho_{\theta}^{\alpha_{ij}}T_{\theta}^{\beta_{ij}}}\bigg|\\
&\leq C \bigg|\frac{\mathcal{P}_{ij}(\sqrt{C_q M_0e^{\nu^{max}_{f_0}t}},C_q M_0e^{\nu^{max}_{f_0}t},C_q M_0e^{\nu^{max}_{f_0}t})}{(C_0e^{-\nu^{max}_{f_0}t})^{\alpha_{ij}}(C_qM_0^{-\frac{5}{3}}e^{-\frac{7}{3}\nu^{max}_{f_0}t})^{\beta_{ij}}}\bigg| \cr
&\leq C_q M_0^ne^{C\nu^{max}_{f_0}t},
	\end{split}
\end{align}
for generic positive constants $n>1$, $C$ and $C_q$. Recall the definition \eqref{A12} of $A_i(\theta)$, which comes from the linearization of $\nu$. From Lemma \ref{4.theta}, we have the following bounds:
\begin{align} \label{A12 bound}
\begin{split}
A_i(\theta)&\leq C_q M_0^ne^{C\nu^{max}_{f_0}t}, \quad \mbox{for} \quad i=1,\cdots,5.
\end{split}
\end{align}
Combining the definition \eqref{def gamma} of $\Gamma(f)$, \eqref{second derv est}, \eqref{A12 bound}, and
\begin{align*}
\left \vert \int_{\mathbb{R}^3} f(t,x,u) (1,u,|u|^2) du \right \vert \leq C_q \| f(t)\|_{L^{\infty,q}_{x,v}},
\end{align*}
we obtain the desired result.
\end{proof}

\subsection{Global decay estimate}
Multiplying \eqref{expan} by $\langle v \rangle^q$ for $q>10$, we obtain
\begin{align}\label{reform bgk}
	\partial_t(\langle v \rangle^qf) +v\cdot \nabla_x (\langle v \rangle^qf) +\langle v \rangle^qf = \langle v \rangle^q(\mathbf{P}f + \Gamma(f)),
\end{align}
where we defined $\mathbf{P}f$ and $\Gamma(f)$ in \eqref{pf} and \eqref{def gamma}, respectively.
\begin{proposition} \label{L^infty esti}
Let $f(t,x,v)$ be the solution of the reformulated BGK equation \eqref{reform bgk} with initial data $f_0$ satisfying \eqref{initial LB} and $\Vert f_0 \Vert_{L^{\infty,q}_{x,v}} \leq M_0$ for $q>10$. Under a priori assumption \eqref{Assumption},
it holds that
\begin{align} \label{main esti}
\begin{split}
\Vert f(t)\Vert_{L^{\infty,q}_{x,v}} &\leq C_{q} e^{-t/2} \left ( 1+\int_0^t \Vert f(s) \Vert_{L^{\infty,q}_{x,v}}ds\right)e^{\nu_{f_0}^{max} t_{eq}}\Vert f_0 \Vert_{L^{\infty,q}_{x,v}}^{n+1}  + C_q\left(1+ C_qM_0^ne^{C\nu_{f_0}^{max}t_{eq}}\overbar{M}\right)^2 \cr
& \quad \times \left(N^5(1-e^{-\lambda}) \overbar{M}+\frac{1}{N^{q-5}}+\frac{\overbar{M}}{N^{q-10}} + N^6 (\lambda^{-2}+N^3) \left(\mathcal{E}(F_0)+N^{\frac{3}{2}}\sqrt{\mathcal{E}(F_0)} \right) \right),
\end{split}
\end{align}
for some generic positive constants $n>1$ and $t\in [0,t_*]$, where $t_{eq}$, $\nu_{f_0}^{max}$, and initial relative entropy $\mathcal{E}(F_0)$ were defined in \eqref{def teq}, \eqref{numax}, and \eqref{relaentp}, respectively.
\end{proposition}
\begin{proof}
Applying Duhamel's principle to 
\eqref{reform bgk}, we have
\begin{align} \label{first iteration}
	\begin{split}
	\vert \langle v \rangle^qf(t,x,v)\vert &\leq e^{-t}\Vert f_0 \Vert_{L^{\infty,q}_{x.v}}\\
	&\quad +\int_0^t e^{-(t-s)}\langle v \rangle^q\left \vert \mathbf{P}f(s,x-v(t-s),v)+\Gamma(f)(s,x-v(t-s),v)\right \vert ds.
	\end{split}
\end{align}
We split the estimate as follows:
\begin{align*}
I_1&= \int_0^t e^{-(t-s)}\langle v \rangle^q\left \vert \mathbf{P}f (s,x-v(t-s),v) \right \vert ds,\cr
I_2&=\int_0^t e^{-(t-s)}\langle v \rangle^q\left\vert \Gamma(f)(s,x-v(t-s),v)\right \vert ds.
\end{align*}
By the definition of $\mathbf{P}f$ in \eqref{pf}, we have
\begin{align} \label{I_1}
	\begin{split}
	I_1 &\leq \int_0^t e^{-(t-s)} \int_{\R^3} \vert f(s,x-v(t-s),u)\vert (1+\vert u \vert^2) duds. 
	\end{split}
\end{align}
For $I_2$, we split the time-integration region into $[0,t_{eq} ]$ and $[t_{eq},t]$:
\begin{align*}
	I_2&=\left ( \int_0^{t_{eq}} + \int_{t_{eq}}^t\right) e^{-(t-s)}\langle v \rangle^q\left\vert \Gamma(f)(s,x-v(t-s),v)\right \vert ds\cr
	&:=I_{2,1}+I_{2,2}.
\end{align*}
For $0\leq s \leq t_{eq}$, by Lemma \ref{nonlinM0}, it holds that
\begin{align} \label{I_21}
	I_{2,1} \leq C_qM_0^ne^{C\nu_{f_0}^{max}t_{eq}}\int_0^{t_{eq}} e^{-(t-s)} \Vert f(s)\Vert_{L^{\infty,q}_{x,v}} \int_{\R^3} \langle u \rangle^2 \vert f(s,x-v(t-s),u)\vert duds.
\end{align}
Recall that we controlled the macroscopic fields $|\rho-1|~, |U|,~ |T-1| \leq 2\delta$ for $t\geq t_{eq}$ in Lemma \ref{macrodelta}. Hence, using the same argument as in Lemma \ref{4.theta}, we obtain $|\rho_{\theta}-1|~, |U_{\theta}|,~ |T_{\theta}-1| \leq 2\delta$ when $t\geq t_{eq}$. This, combined with the argument in the proof of Lemma \ref{nonlinM0}, yields
\begin{align*}
	\langle v \rangle^q\Gamma(f)(t,x,v) \leq  C\Vert f(t)\Vert_{L^{\infty,q}_{x,v}}\int_{\mathbb{R}^3} \vert f(u) \vert (1+|u|^2) du, \quad \mbox{for} \quad t\geq t_{eq}.
\end{align*}
Thus, $I_{2,2}$ can be further bounded by
\begin{align} \label{I_22}
	I_{2,2} \leq C \int_{t_{eq}}^t e^{-(t-s)} \Vert f(s)\Vert_{L^{\infty,q}_{x,v}} \int_{\R^3} \langle u \rangle^2 \vert f(s,x-v(t-s),u)\vert duds.
\end{align}
Applying \eqref{I_1}, \eqref{I_21}  and \eqref{I_22} on \eqref{first iteration}, we obtain
\begin{align} \label{Duhamel arg}
	\begin{split}
	\vert \langle v \rangle^qf(t,x,v) \vert &\leq e^{-t} \Vert f_0 \Vert_{L^{\infty,q}_{x,v}}  + C\int_0^t e^{-(t-s)} \int_{\R^3} \langle u \rangle^2 \vert f(s,x-v(t-s),u)\vert duds\\
	&\quad +C_qM_0^ne^{C\nu_{f_0}^{max}t_{eq}} \int_0^t e^{-(t-s)} \Vert f(s)\Vert_{L^{\infty,q}_{x,v}} \int_{\R^3} \langle u \rangle^2 \vert f(s,x-v(t-s),u)\vert duds.
	\end{split}
\end{align}
We denote
\begin{align*}
B:= \int_{\R^3} \langle u \rangle^2 \vert f(s,x-v(t-s),u)\vert du,
\end{align*}
and split the integral region as $\{\vert u \vert \geq N\}$ and $\{\vert u \vert \leq N\}$ as follows:
\begin{align*}
B_1= \int_{\vert u \vert \geq N} \langle u \rangle^2 \vert f(s,x-v(t-s),u)\vert du, \cr
B_2= \int_{\vert u \vert \leq N} \langle u \rangle^2 \vert f(s,x-v(t-s),u)\vert du.
\end{align*}
Over $\{\vert u \vert \geq N\}$, it holds that
\begin{align}\label{B1}
B_1 \leq \int_{\vert u \vert \geq N} \langle u \rangle^{2-q}\vert \langle u \rangle^qf(s,x-v(t-s),u)\vert du \leq C_q \frac{\Vert f(s) \Vert_{L^{\infty,q}_{x,v}}}{N^{q-5}} \leq C_q \frac{\overbar{M}}{N^{q-5}} .
\end{align}
Applying \eqref{Duhamel arg} again to the integrand $\vert \langle u \rangle^qf(s,x-v(t-s),u)\vert$ on $B_2$, we get
\begin{align}\label{esti2}
\begin{split}
B_2 
&\leq C_qe^{-s} \Vert f_0 \Vert_{L^{\infty,q}_{x,v}} \\
&\quad + C\int_0^s e^{-(s-s')}\int_{\vert u \vert \leq N}\langle u \rangle^{2-q}\int_{\R^3} \langle u' \rangle^{2-q}\vert  \langle u' \rangle^qf(s',x-v(t-s)-u(s-s'),u')\vert du'duds'\\
&\quad +C_qM_0^ne^{C\nu_{f_0}^{max}t_{eq}} \int_0^s e^{-(s-s')} \Vert f(s') \Vert_{L^{\infty,q}_{x,v}}\int_{\vert u \vert \leq N}\langle u \rangle^{2-q}\\
&\hspace{3.4cm} \times \int_{\R^3} \langle u' \rangle^{2-q}\vert  \langle u' \rangle^q f(s',x-v(t-s)-u(s-s'),u')\vert du'duds'.
\end{split}
\end{align}
For the integral $\int du'du$ part, applying the same argument as in Lemma \ref{Pfesti}, we have
\begin{align}\label{samearg}
\begin{split}
\int_0^s e^{-(s-s')} \int_{\vert u \vert \leq N}\langle u \rangle^{2-q}\int_{\R^3} \langle u' \rangle^{2-q}\vert   \langle u' \rangle^qf(s',x-v(t-s)-u(s-s'),u')\vert du'duds' \cr
\leq C_qN^5(1-e^{-\lambda}) \overbar{M}+\frac{C_q}{N^{q-10}}\overbar{M} + CN^6 (\lambda^{-2}+N^3) \left(\mathcal{E}(F_0)+N^{\frac{3}{2}}\sqrt{\mathcal{E}(F_0)} \right).
\end{split}
\end{align}
Substituting the estimate \eqref{samearg} in \eqref{esti2} yields
\begin{align} \label{B2}
\begin{split}
B_2&\leq C_qe^{-s} \Vert f_0 \Vert_{L^{\infty,q}_{x,v}}+ C_q\left(1+ \sup_{0\leq s' \leq s} \Vert f(s') \Vert_{L^{\infty,q}_{x,v}}\right) \cr
&\quad \times \left(N^5(1-e^{-\lambda}) \overbar{M}+\frac{\overbar{M}}{N^{q-10}} +N^6 (\lambda^{-2}+N^3) \left(\mathcal{E}(F_0)+N^{\frac{3}{2}}\sqrt{\mathcal{E}(F_0)} \right) \right).
\end{split}
\end{align}
We combine \eqref{B1} and \eqref{B2} to obtain
\begin{align}\label{barB}
B& \leq C_qe^{-s} \Vert f_0 \Vert_{L^{\infty,q}_{x,v}} + \bar{B}\left(N,\lambda,\mathcal{E}(F_0),\overbar{M}\right),
\end{align}
where
\begin{align*}
\bar{B}(N,\lambda,&\mathcal{E}(F_0),\overbar{M}):= C_q\left(1+ \overbar{M}\right) \cr
& \times \left(N^5(1-e^{-\lambda}) \overbar{M}+\frac{1}{N^{q-5}}+\frac{\overbar{M}}{N^{q-10}} + N^6 (\lambda^{-2}+N^3) \left(\mathcal{E}(F_0)+N^{\frac{3}{2}}\sqrt{\mathcal{E}(F_0)} \right) \right).
\end{align*}
We substitute \eqref{barB} in \eqref{Duhamel arg} to have
\begin{align*}
\vert \langle v \rangle^qf(t,x,v) \vert &\leq e^{-t} \Vert f_0 \Vert_{L^{\infty,q}_{x,v}}  \cr
&\quad  + C_{q} e^{-t/2} \Vert f_0 \Vert_{L^{\infty,q}_{x,v}} + \bar{B}\\
&\quad +C_qM_0^ne^{C\nu_{f_0}^{max}t_{eq}} \left(e^{-t}\int_0^t \Vert f(s)\Vert_{L^{\infty,q}_{x,v}} ds\Vert f_0 \Vert_{L^{\infty,q}_{x,v}} + \sup_{0\leq s \leq t} \Vert f(s) \Vert_{L^{\infty,q}_{x,v}} \bar{B}\right).
\end{align*}
Therefore, we obtain
\begin{align*}
\vert \langle v \rangle^qf(t,x,v) \vert
&\leq C_qe^{-t/2}\left(1+\int_0^t \Vert f(s)\Vert_{L^{\infty,q}_{x,v}} ds\right) M_0^{n+1}e^{C\nu_{f_0}^{max}t_{eq}} + (1+C_qM_0^ne^{C\nu_{f_0}^{max}t_{eq}}\overbar{M}) \bar{B}.
\end{align*}
Finally, we complete the proof of Proposition \ref{L^infty esti} by taking $L^\infty_{x,v}$-norm above.
\end{proof}

\subsection{Proof of main theorem}
\begin{proof} [Proof of Theorem \ref{maintheorem}]
For convenience of notation, we rewrite \eqref{main esti} in Proposition \ref{L^infty esti} as
\begin{align} \label{main esti 3}
	\Vert f(t) \Vert_{L^{\infty,q}_{x,v}} \leq C_q e^{\nu_{f_0}^{max}t_{eq}} M_0^{n+1} e^{-t/2} \left( 1+ \int_0^t \Vert f(s) \Vert_{L^{\infty,q}_{x,v}} ds \right) +D,
\end{align}
where 
\begin{align} \label{def D}
	\begin{split}
	D&:= C_q\left(1+ C_qM_0^ne^{C\nu_{f_0}^{max}t_{eq}} \overbar{M}\right)^2 \cr
	&\quad  \times \left(N^5(1-e^{-\lambda}) \overbar{M}+\frac{1}{N^{q-5}}+\frac{\overbar{M}}{N^{q-10}} + N^6 (\lambda^{-2}+N^3) \left(\mathcal{E}(F_0)+N^{\frac{3}{2}}\sqrt{\mathcal{E}(F_0)} \right) \right).
	\end{split}
\end{align}
If we define
\begin{align*}
	Y(t):= 1+ \int_0^t \Vert f(s) \Vert_{L^{\infty,q}_{x,v}} ds,
\end{align*}
then we directly deduced from \eqref{main esti 3} that
\begin{align*}
	Y'(t) \leq  C_q e^{\nu_{f_0}^{max}t_{eq}} M_0^{n+1} e^{-t/2} Y(t) +D.
\end{align*}
By multiplying both sides above by $\exp\left\{ -2 C_q e^{\nu_{f_0}^{max}t_{eq}}M_0^{n+1}(1-e^{-t/2})\right\}$, we have
\begin{align*}
	\left( Y(t) \exp\left \{ -2 C_q e^{\nu_{f_0}^{max}t_{eq}}M_0^{n+1}(1-e^{-t/2})\right\}\right)' \leq D\exp\left\{ -2 C_q e^{\nu_{f_0}^{max}t_{eq}}M_0^{n+1}(1-e^{-t/2})\right\}\leq D,
\end{align*}
for all $0\leq t \leq t_*$. Taking the time integration over $[0,t]$, one obtains that
\begin{align} \label{main proof 1}
	\begin{split}
	Y(t) &\leq (1+Dt)\exp\left \{ 2C_q e^{\nu_{f_0}^{max}t_{eq}}M_0^{n+1}(1-e^{-t/2})\right\} \\
	&\leq (1+Dt)\exp\left \{ 2C_q e^{\nu_{f_0}^{max}t_{eq}}M_0^{n+1}\right\}.
	\end{split}
\end{align}
And then, substituting \eqref{main proof 1} into \eqref{main esti 3}, it holds that
\begin{align} \label{main proof 2}
	\Vert f(t) \Vert_{L^{\infty,q}_{x,v}} \leq C_q e^{\nu_{f_0}^{max}t_{eq}} M_0^{n+1}\exp\left \{ 2C_q e^{\nu_{f_0}^{max}t_{eq}}M_0^{n+1}\right\} (1+Dt) e^{-t/2}+D,
\end{align}
for all $0\leq t \leq t_{*}$. We now define
\begin{align} \label{M bar}
	\overbar{M}:= 4C_qM_0^{n+1}\exp\left \{\nu_{f_0}^{max}t_{eq}+ 2C_q e^{\nu_{f_0}^{max}t_{eq}}M_0^{n+1}\right\},
\end{align}
and
\begin{align} \label{t star}
	t_{*}:= 4 \left[ \ln \overbar{M} - \ln \delta  \right],
\end{align}
for $0<\delta<1$. From \eqref{main proof 2} and the definition \eqref{M bar} of $\overbar{M}$, we have
\begin{align} \label{main proof 3}
	\Vert f(t) \Vert_{L^{\infty,q}_{x,v}} \leq \frac{1}{4} \overbar{M} (1+Dt)e^{-t/2} +D \leq \frac{1}{4}\overbar{M}\left[1+2D\right]e^{-t/4}+D,
\end{align}
where we used $te^{-t/4}\leq 2$. Recall the definition \eqref{def D} of $D$. We first take $N=N(\overbar{M})>0$ large enough, then $\lambda=\lambda(N,\overbar{M})>0$ sufficiently small, and finally let $\mathcal{E}(F_0)\leq \varepsilon=\varepsilon(\delta,\lambda,N,\overbar{M})>0$ sufficiently small, so that
\begin{align*}
	D \leq \min\left\{\frac{\overbar{M}}{8},\frac{1}{4},\frac{\delta}{4} \right\}.
\end{align*}
Hence, it follows from \eqref{main proof 3} that
\begin{align} \label{closed apriori}
	\Vert f(t) \Vert_{L^{\infty,q}_{x,v}} \leq \frac{3}{8} \overbar{M} + \frac{1}{8}\overbar{M} \leq \frac{1}{2}\overbar{M},
\end{align}
for all $0\leq t \leq t_*$. Since $\overbar{M}$ depends on $M_0$ and $\delta$, the parameter $\epsilon$ also depends only on $M_0$ and $\delta$· Under $\mathcal{E}(F_0) \leq \varepsilon=\varepsilon(\delta,M_0)$, we have shown that a priori assumption \eqref{Assumption} is closed. \\

The next step is to extend the BGK solution to time interval $t\in[0,t_*]$ by using \eqref{closed apriori} and Proposition \ref{local}. Firstly, through Proposition \ref{local}, the solution of the BGK equation $f(t)$ exists on $t\in[0, t_0]$ satisfying
\begin{align*} 
	\sup_{0\leq t \leq t_0} \Vert f(t) \Vert_{L^{\infty,q}_{x,v}} \leq 2 \Vert f_0 \Vert_{L^{\infty,q}_{x,v}} \leq \frac{1}{2} \overbar{M}.
\end{align*}
We set $t_0$ as an initial time. Then Proposition \ref{local} gives the local existence time $\tilde{t}=t_0\left(\overbar{M}/2\right)$ satisfying
\begin{align*}
	\sup_{t_0 \leq t \leq t_0+\tilde{t}} \Vert f(t)\Vert_{L^{\infty,q}_{x,v}} \leq 2 \Vert f(t_0) \Vert_{L^{\infty,q}_{x,v}} \leq \overbar{M},
\end{align*}
when the initial data starts with $\Vert f(t_0) \Vert_{L^{\infty,q}_{x,v}} \leq \frac{1}{2} \overbar{M}$. Note that the a priori assumption holds for $t \in [0,t_0 + \tilde{t}]$. Hence, we can apply the estimate \eqref{closed apriori}, and then the BGK solution $f(t)$ has the following bound
\begin{align*}
	\sup_{0\leq t \leq t_0+\tilde{t}} \Vert f(t)\Vert_{L^{\infty,q}_{x,v}} \leq \frac{1}{2} \overbar{M}.
\end{align*}
Repeating the procedure until $t_*$, the BGK solution $f(t)$ exists and is unique on $t\in[0,t_*]$ and satisfies \eqref{closed apriori}. By definition \eqref{t star} of $t_*$ and the estimate \eqref{main proof 3}, we obtain that
\begin{align*}
	\Vert f(t_*) \Vert_{L^{\infty,q}_{x,v}} \leq \frac{3}{8}\delta + \frac{1}{4}\delta < \delta,
\end{align*}
due to $D \leq \delta/4$. The final step is to extend the BGK solution to $[t_*,\infty]$. From Proposition \ref{smallthm} in Section 6, we prove the global well-posedness and exponential decay of the BGK solution with small initial data $\Vert f_0 \Vert_{L^{\infty,q}_{x,v}}\leq \delta$. Hence, if we treat $\Vert f(t_*)\Vert_{L^{\infty,q}_{x,v}}$ as initial data in Proposition \ref{smallthm}, then it follows from (3) in Proposition \ref{smallthm} that
\begin{align*}
	\Vert f(t) \Vert_{L^{\infty,q}_{x,v}} \leq C_qe^{-C(t-t_*)}\Vert f(t_*) \Vert_{L^{\infty,q}_{x,v}},
\end{align*}
for all $t \geq t_*$.
\end{proof}

\section{Asymptotic stability for small amplitude regime}\label{secsmall}
In this section, we prove that if the initial $L^{\infty}$ norm is sufficiently small, then there exists a unique non-negative global solution.

\begin{proposition}\label{smallthm} Let $f_0$ satisfy the conservation laws \eqref{normal}. There exists $\delta > 0$ such that if $\|f_0 \|_{L^{\infty,q}_{x,v}}\leq \delta$, then there exists a unique global solution of the BGK model \eqref{expan}. Moreover, the following holds:
\begin{enumerate}
\item The solution $f(t,x,v)$ satisfies the conservation laws
	\begin{align*}
	\int_{\mathbb{T}^{3}\times \R^{3}} f (t,x,v)
	(1 ,  v , |v|^{2})dvdx = 0.
	\end{align*}
\item The solution is non-negative: $F(t,x,v)=\mu(v)+f(t,x,v)\geq 0$.
\item The perturbation decays exponentially:
\begin{align*}
\| f(t) \|_{L^{\infty,q}_{x,v}} &\leq C_q\delta e^{-k t},
\end{align*}
for positive constants $k$ and $C_q$.
\item Let $f$ and $\tilde{f}$ be solutions corresponding to the initial data $f_0$ and $\tilde{f}_0$, respectively. Then
\begin{align*}
\sup_{s\in[0,t]}e^{ks}\|(f-\tilde{f})(s)\|_{L^{\infty,q}_{x,v}}
&\leq C_q\|f_0-\tilde{f}_0\|_{L^{\infty,q}_{x,v}},
\end{align*}
for positive constants $k$ and $C_q$.
\end{enumerate}
\end{proposition}

To prove the above proposition, we follow the argument of \cite{Guo-Briant} where the asymptotic stability of the Boltzmann equation (with some boundary condition) is proved in $L^{\infty}$ with small initial data.
We decompose \eqref{expan} into the following two equations:
\begin{align*}
&\partial_t f_1 + v\cdot\nabla_x f_1 +  f_1 = \Gamma(f_1+f_2),  \qquad f_1(0,x,v)=f_0(x,v),  \\
&\partial_t f_2 + v\cdot\nabla_x f_2 = (\mathbf{P}f_2-f_2) + \mathbf{P}f_1,  \qquad f_2(0,x,v)=0,
\end{align*}
where $f = f_1 + f_2$. In the following, we study the above two equations to derive the existence and asymptotic behavior of $f_1$ and $f_2$.

\begin{lemma}\label{nonlinsubt}
	There exists $\delta > 0$ such that if $\|f(t)\|_{L^{\infty,q}_{x,v}}\leq \delta$ and $\|g(t)\|_{L^{\infty,q}_{x,v}}\leq \delta$, then we have
	\begin{align*}
	(1)~& \| \Gamma(f)(t)\|_{L^{\infty,q}_{x,v}} \leq C\|f(t)\|_{L^{\infty,q}_{x,v}}^2,  \\
	(2)~& \| (\Gamma(f)-\Gamma(g))(t) \|_{L^{\infty,q}_{x,v}} \leq C_{\delta}\|(f-g)(t)\|_{L^{\infty,q}_{x,v}},
	\end{align*}
	for positive constants $C$ and $0<C_{\delta}<1$.
\end{lemma}
\begin{proof}
	(1) By Lemma \ref{macrodelta}, the assumption $\|f\|_{L^{\infty,q}_{x,v}}\leq \delta$ implies
	\begin{align*}
	|\rho-1|,~|U|,~|T-1| \leq 2\delta.
	\end{align*}
	So that, for a sufficiently small $\delta$, the nonlinear term $\Gamma_1$ can be reduced to
	\begin{align*}
	\big|\langle v \rangle^q \Gamma_1(f)\big| &\leq C\big|\langle v \rangle^q(\mathbf{P}f-f)\big| \Big|\int_{\mathbb{R}^3} (1+|v|^2)f  dv\Big| \leq C\|f\|_{L^{\infty,q}_{x,v}}^2.
	\end{align*}
	Applying a similar process to Lemma \ref{nonlinM0}, we also have
	\begin{align*}
	|\langle v \rangle^q\Gamma_2(f)| &= \rho^a T^b\sum_{1\leq i,j\leq 5}\int_0^1 \langle v \rangle^q \frac{\mathcal{P}_{ij}((v-U_{\theta}),U_{\theta},T_{\theta})}{\rho_{\theta}^{\alpha_{ij}}T_{\theta}^{\beta_{ij}}}\mathcal{M}(\theta)  (1-\theta) d\theta \int_{\mathbb{R}^3}f e_i dv \int_{\mathbb{R}^3}f e_jdv \cr
	&\leq C\|f\|_{L^{\infty,q}_{x,v}}^2.
	\end{align*}
(2) Since we have $|\rho-1|,~|U|,~|T-1| \leq 2\delta$ for $f$ and $g$, applying the same argument as in Lemma \ref{nonlinfg} with
\begin{align*}
\int_{\mathbb{R}^3} f e_i dv \leq C\int_{\mathbb{R}^3} (1+|v|+|v|^2)|f| dv \leq C_q \|f\|_{L^{\infty,q}_{x,v}}\leq C_q \delta , \quad \mbox{for} \quad i=1,\cdots,5,
\end{align*}
we can obtain $C_{\delta}<1$ satisfying
\begin{align*}
\| \Gamma(f)-\Gamma(g) \|_{L^{\infty,q}_{x,v}} \leq C_{\delta}\|f-g\|_{L^{\infty,q}_{x,v}},
\end{align*}
for sufficiently small $\delta$.
\end{proof}

\begin{lemma}\label{f1lem} There exists $\delta >0$ such that if $\|f_0\|_{L^{\infty,q}_{x,v}}\leq \delta$  and $\|g(t)\|_{L^{\infty,q}_{x,v}}\leq C_q\delta  e^{-(1-\epsilon)t}$ for any $\epsilon\in(0,1)$, then there exists a solution $f_1$ to the following equation
\begin{align*}
&\partial_t f_1 + v\cdot\nabla_x f_1 +  f_1 = \Gamma(f_1+g),  \qquad f_1(0,x,v)=f_0(x,v),
\end{align*}
satisfying
\begin{align*}
\|f_1(t)\|_{L^{\infty,q}_{x,v}} \leq e^{-(1-\epsilon)t}(\|f_0\|_{L^{\infty,q}_{x,v}} + \delta).
\end{align*}
\end{lemma}
\begin{proof}
We define the following iteration for $f_1^n$ starting with $f_1^0(t,x,v)=0$
\begin{align*}
&\partial_t f_1^{n+1} + v\cdot\nabla_x f_1^{n+1} +  f_1^{n+1} = \Gamma(f_1^n+g),  \qquad f_1(0,x,v)=f_0(x,v).
\end{align*}
Then we prove that $\{f_1^n\}_{n\geq0}$ is uniformly bounded and Cauchy. We write the equation in the mild form:
\begin{align}\label{mildf1}
f_1^{n+1}(t,x,v) = e^{-t}f_0(x-vt,v)+ \int_0^t  e^{-(t-s)} \Gamma(f_1^n+g)(s,x-v(t-s),v) ds.
\end{align}
Taking $L^{\infty,q}_{x,v}$ on \eqref{mildf1} yields
\begin{align*}
\|f_1^{n+1}(t)\|_{L^{\infty,q}_{x,v}} &\leq e^{-t}\|f_0\|_{L^{\infty,q}_{x,v}}+ \int_0^t  e^{-(t-s)} \|\Gamma(f_1^n+g)(s)\|_{L^{\infty,q}_{x,v}} ds.
\end{align*}
We multiply $e^{(1-\epsilon)t}$ on both sides and use Lemma \ref{nonlinsubt} to get
\begin{align*}
e^{(1-\epsilon)t}\|f_1^{n+1}(t)\|_{L^{\infty,q}_{x,v}} &\leq e^{-\epsilon t}\|f_0\|_{L^{\infty,q}_{x,v}}+ \int_0^t  e^{(1-\epsilon)s}e^{-\epsilon(t-s)} \|\Gamma(f_1^n+g)(s)\|_{L^{\infty,q}_{x,v}} ds \cr
&\leq \|f_0\|_{L^{\infty,q}_{x,v}}+C\sup_{s\in[0,t]} e^{(1-\epsilon)s}\|(f_1^n+g)(s)\|_{L^{\infty,q}_{x,v}}^2. 
\end{align*}
Taking supremum on each side, we have
\begin{align*}
\sup_{s\in[0,t]}e^{(1-\epsilon)s}\|f_1^{n+1}(s)\|_{L^{\infty,q}_{x,v}} &\leq \|f_0\|_{L^{\infty,q}_{x,v}} + C\sup_{s\in[0,t]} e^{(1-\epsilon)s}\left(\|f_1^n(s)\|_{L^{\infty,q}_{x,v}}^2\right) +CC_q^2\delta^2.
\end{align*}
Thus, if the $n$-th step has the following bound,
\begin{align*}
\sup_{s\in[0,t]}e^{(1-\epsilon)s}\|f_1^n(s)\|_{L^{\infty,q}_{x,v}}  \leq \|f_0\|_{L^{\infty,q}_{x,v}} + \delta,
\end{align*}
then the $(n+1)$-th step satisfies
\begin{align*}
\sup_{s\in[0,t]}e^{(1-\epsilon)s}\|f_1^{n+1}(s)\|_{L^{\infty,q}_{x,v}}&= \|f_0\|_{L^{\infty,q}_{x,v}} + C(2\delta)^2+CC_q^2\delta^2 \leq \|f_0\|_{L^{\infty,q}_{x,v}} + \delta,
\end{align*}
for sufficiently small $\delta$ satisfying $C(4+C_q^2)\delta^2 \leq \delta$. This gives the desired uniform boundedness:
\begin{align*}
\sup_{s\in[0,t]}e^{(1-\epsilon)s}\|f_1^n(s)\|_{L^{\infty,q}_{x,v}}  \leq \|f_0\|_{L^{\infty,q}_{x,v}} + \delta.
\end{align*}
To prove that $\{f_1^n\}$ is a Cauchy sequence, we consider the difference between $f^{n+1}$ and $f^n$:
\begin{align*}
e^t(f_1^{n+1}-f_1^n)(t,x,v) = \int_0^t  e^s \left[\Gamma(f_1^n+g)-\Gamma(f_1^{n-1}+g)\right](s,x-v(t-s),v) ds.
\end{align*}
Since, for sufficiently small $\delta$, we have for all $n\geq 0 $
\begin{align*}
\int_{\mathbb{R}^3}(1,v,|v|^2)(f_1^n+g) dv \leq \|f_1^n\|_{L^{\infty,q}_{x,v}} + \|g\|_{L^{\infty,q}_{x,v}} \leq (C_q+1)\delta,
\end{align*}
we can employ Lemma \ref{nonlinsubt} (2) to get
\begin{align*}
\| \Gamma(f^n+g)-\Gamma(f^{n-1}+g) \|_{L^{\infty,q}_{x,v}} \leq C_{\delta}\|f^n-f^{n-1}\|_{L^{\infty,q}_{x,v}},
\end{align*}
for $n\geq 1 $. Therefore we have
\begin{align*}
\sup_{s\in[0,t]}e^{(1-\epsilon)s}\|(f_1^{n+1}-f_1^n)(s)\|_{L^{\infty,q}_{x,v}} & \leq C_{\delta} \sup_{s\in[0,t]}e^{(1-\epsilon)s} \|(f_1^n-f_1^{n-1})(s)\|_{L^{\infty,q}_{x,v}},
\end{align*}
for $0<C_{\delta}<1$. This completes the proof.
\end{proof}

\hide
{\color{blue}
\begin{lemma} \label{St group}
	Let $f$ solve
	\begin{eqnarray}
	\begin{split}
		(\p_{t} + v\cdot\nabla_{x} +1)f = \mathbf{P}f,\quad \Pi f = 0,
	\end{split}
	\end{eqnarray}
	with initial data  $f(0) = f_0$ which satisfies compability condition $\Pi f_0 = 0$. Then solution $f(t)$ is written by semi-group $S(t)f_0$ and
	\begin{equation*}
	\begin{split}
		\|S(t)f_0\|_{L_{x,v}^{\infty,q}(\mu^{-\zeta})} &\leq e^{-t} \|f_0\|_{L_{x,v}^{\infty,q}(\mu^{-\zeta})}.
	\end{split}
	\end{equation*}
\end{lemma}
} 
\unhide

Before we proceed to the next lemma, we define 
\begin{align*}
\Pi(f) &:= \int_{\mathbb{T}^3} \mathbf{P}f dx \cr
&=\int_{\mathbb{T}^3\times\mathbb{R}^3}fdvdx \mu+\int_{\mathbb{T}^3\times\mathbb{R}^3}fvdvdx\cdot (v\mu) +\int_{\mathbb{T}^3\times\mathbb{R}^3}f\frac{|v|^2-3}{\sqrt{6}}dvdx\left(\frac{|v|^2-3}{\sqrt{6}}\mu\right).
\end{align*}
Note that, unlike projection operator $\mathbf{P}$, $\Pi$ does commutes with transport $v\cdot \nabla_x$.

\begin{lemma}\label{f2lem} Let $\|g\|_{L_t^{\infty}L^{\infty,q}_{x,v}} < \infty$. Then there exists a unique solution $f_2\in L_t^{\infty}L^{\infty}_{x,v}(\mu^{-\zeta})$ to
\begin{align} \label{f2 eq}
\partial_t f_2 + v\cdot\nabla_x f_2 = (\mathbf{P}f_2-f_2) + \mathbf{P}g,  \qquad f_2(0,x,v)=0,
\end{align}
where $\zeta \in [0,1)$. Moreover, if $\Pi(f_2+g)=0$ and $\|g(t)\|_{L^{\infty,q}_{x,v}}\leq e^{-(1-\epsilon)t}(\|f_0\|_{L^{\infty,q}_{x,v}} + \delta)$ for $\epsilon\in(1-\eta_s,1)$ ($\eta_s$ will be determined in the proof), then we have
\begin{align*}
\|f_2(t)\|_{L^{\infty}_{x,v}(\mu^{-\zeta})}\leq C_q\delta e^{-(1-\epsilon)t}.
\end{align*}
\end{lemma}
\begin{proof}
Let $S(t)$ be the semi-group so that $S(t)f_0$ solves the following equation:
\begin{equation}\label{linearf2}
 (\p_{t} + v\cdot\nabla_{x} +1)f = \mathbf{P}f,\qquad f(0) = f_0 ,\qquad \Pi f_0 = 0.
\end{equation}
We first consider the $L^{\infty}$ decay of $S(t)$. We write \eqref{linearf2} in the mild form
\begin{align*}
f(t,x,v) = e^{-t}f_0(x-vt,v) + \int_0^t e^{-(t-s)}\mathbf{P}f(s,x-v(t-s),v)ds.
\end{align*}
Multiplying $\mu^{-\zeta}$ and applying double iteration on $\mathbf{P}f$, we have
\begin{align*}
\|f(t)\|_{L^{\infty}_{x,v}(\mu^{-\zeta})} \lesssim  e^{-\frac{1}{2}t}\|f_0\|_{L^{\infty}_{x,v}(\mu^{-\zeta})} +C_{T_0} \int_0^t \|f(s)\|_{L^2_{x,v}(\mu^{-1/2})}ds, \quad \mbox{for} \quad t\in[0,T_0],
\end{align*}
where we used $\|\mathbf{P}f\|_{L^{\infty}_{x,v}(\mu^{-\zeta})} \leq C\|f\|_{L^{2}_{x,v}(\mu^{-1/2})} $. Recalling that there is $\eta$ such that $\|f(t)\|_{L^2_{x,v}(\mu^{-1/2})}\leq Ce^{-\eta t}\|f_0\|_{L^2_{x,v}(\mu^{-1/2})}$ (See \cite{Yun1,Yun3}), we have from Proposition 5.4 in \cite{Guo-Briant} that
\begin{align}\label{St est}
\|S(t)f_0\|_{L^{\infty}_{x,v}(\mu^{-\zeta})} \leq  e^{-\eta_s t}\|f_0\|_{L^{\infty}_{x,v}(\mu^{-\zeta})},
\end{align}
for $0<\eta_s<1$. Now, we consider the estimate of $f_2$ in \eqref{f2 eq}. From $\Pi(f_2+g)=0$, we can estimate $\Pi f_2$ as follows:
\begin{align}\label{Pipart}
\| \Pi f_2(t) \|_{L^{\infty}_{x,v}(\mu^{-\zeta})} &= \|\Pi g(t) \|_{L^{\infty}_{x,v}(\mu^{-\zeta})} \leq C_q\|g(t) \|_{L^{\infty,q}_{x,v}} \leq C_qe^{-(1-\epsilon)t}(\|f_0\|_{L^{\infty,q}_{x,v}} + \delta).
\end{align}
To estimate $(I-\Pi)f_2$, we rewrite $f_2$ by using the definition of the semi-group $S(t)$:
\begin{align*}
f_2 =\int_0^t  S(t-s)\mathbf{P}g ds.
\end{align*}
We claim that $I-\Pi$ commutes with the semi-group $S(t)$:
\begin{align}\label{claim2}
(I-\Pi)f_2 =\int_0^t  S(t-s)(I-\Pi)\mathbf{P}g ds.
\end{align}
With this claim assumed to be true, we apply \eqref{St est} and \eqref{Pipart} to obtain
\begin{align}\label{Piperppart}
\begin{split}
\| (I-\Pi)f_2(t) \|_{L^{\infty}_{x,v}(\mu^{-\zeta})} &\leq \int_0^t  e^{-\eta_s(t-s)}\| \mathbf{P}g(s) \|_{L^{\infty}_{x,v}(\mu^{-\zeta})} ds \cr
&\leq C_q\int_0^t  e^{-\eta_s(t-s)}e^{-(1-\epsilon)s}(\|f_0\|_{L^{\infty,q}_{x,v}} + \delta) ds \cr
&\leq \frac{C_q}{\eta_s-(1-\epsilon)}e^{-(1-\epsilon)t} (\|f_0\|_{L^{\infty,q}_{x,v}} + \delta), 
\end{split}
\end{align}
where we used $0<1-\epsilon<\eta_s$. Combining \eqref{Pipart} and \eqref{Piperppart} gives the desired result:
\begin{align*}
\| f_2(t) \|_{L^{\infty}_{x,v}(\mu^{-\zeta})} &\leq \| \Pi f_2(t) \|_{L^{\infty}_{x,v}(\mu^{-\zeta})} +\| (I-\Pi)f_2(t) \|_{L^{\infty}_{x,v}(\mu^{-\zeta})} \cr
&\leq C_q e^{-(1-\epsilon)t} (\|f_0\|_{L^{\infty,q}_{x,v}} + \delta).
\end{align*}
Now we go back to the proof of the claim \eqref{claim2}. We first prove that $\Pi$ commutes with $\mathbf{P}$:
\begin{align*}
\mathbf{P}\Pi f&=\mathbf{P}\left(\int_{\mathbb{T}^3\times\mathbb{R}^3}fdvdx \mu+\int_{\mathbb{T}^3\times\mathbb{R}^3}fvdvdx\cdot (v\mu) +\int_{\mathbb{T}^3\times\mathbb{R}^3}f\frac{|v|^2-3}{\sqrt{6}}dvdx\left(\frac{|v|^2-3}{\sqrt{6}}\mu\right)\right) \cr
&= \mathbf{P}(\mu)\int_{\mathbb{T}^3\times\mathbb{R}^3}fdvdx +\mathbf{P}(v\mu)\cdot  \int_{\mathbb{T}^3\times\mathbb{R}^3}fvdvdx +\mathbf{P}\left(\frac{|v|^2-3}{\sqrt{6}}\mu\right)\int_{\mathbb{T}^3\times\mathbb{R}^3}f\frac{|v|^2-3}{\sqrt{6}}dvdx \cr
&=\Pi \mathbf{P} f.
\end{align*}
So that $(I-\Pi)(I-\mathbf{P}) =(I-\mathbf{P})(I-\Pi) $. We can easily check that $\Pi$ also commutes with $v\cdot\nabla_x$ as in \cite{Guo-Briant}. Now, we apply $(I-\Pi)$ on \eqref{f2 eq} and use these commutation relations, we get
\begin{align*}
\partial_t (I-\Pi)f_2 + v\cdot\nabla_x (I-\Pi)f_2 = (P-I)(I-\Pi)f_2 + (I-\Pi)\mathbf{P}g.
\end{align*}
This completes the proof of the claim \eqref{claim2}.
\end{proof}

\begin{proof}[Proof of Proposition \ref{smallthm}]
We define the following iteration:
\begin{align}\label{f1f2}
\begin{split}
&\partial_t f_1^{n+1} + v\cdot\nabla_x f_1^{n+1} +  f_1^{n+1} = \Gamma(f_1^{n+1}+f_2^n),  \qquad f_1^{n+1}(0,x,v)=f_0(x,v), \cr
&\partial_t f_2^{n+1} + v\cdot\nabla_x f_2^{n+1} = (\mathbf{P}f_2^{n+1}-f_2^{n+1}) + \mathbf{P}f_1^{n+1},  \qquad f_2^{n+1}(0,x,v)=0,
\end{split}
\end{align}
for $n\geq0$, start with $f_1^0(t,x,v)=0$, and $f_2^0(t,x,v)=0$. The existence of a solution is guaranteed by Lemma \ref{f1lem} and \ref{f2lem} in the following manner: For the $g=f_2^n$ case, Lemma \ref{f1lem} implies that if $\|f_0\|_{L^{\infty,q}_{x,v}}\leq \delta$  and $\|f_2^n(t)\|_{L^{\infty,q}_{x,v}}\leq C_q\delta e^{-(1-\epsilon)t}$, then there exists solution $f_1^{n+1}$ of equation $\eqref{f1f2}_1$, and $f_1^{n+1}$ satisfies
\begin{align*}
\|f_1^{n+1}(t)\|_{L^{\infty,q}_{x,v}} \leq e^{-(1-\epsilon)t}(\|f_0\|_{L^{\infty,q}_{x,v}} + \delta).
\end{align*}
Similarly, for case $g=f_1^{n+1}$, Lemma \ref{f2lem} implies that if $\|f_1^{n+1}\|_{L_t^{\infty}L^{\infty,q}_{x,v}}\leq \infty$, then there exists a unique solution $f_2^{n+1}$ of $\eqref{f1f2}_2$ in $L_t^{\infty}L^{\infty}_{x,v}(\mu^{-\zeta})$.
Moreover, if $\Pi(f_2^{n+1}+f_1^{n+1})=0$ and $\|f_1^{n+1}(t)\|_{L^{\infty,q}_{x,v}}\leq e^{-(1-\epsilon)t}(\|f_0\|_{L^{\infty,q}_{x,v}} + \delta)$, then we have
\begin{align*}
\|f_2^{n+1}(t)\|_{L^{\infty}_{x,v}(\mu^{-\zeta})}\leq C_q\delta e^{-(1-\epsilon)t}.
\end{align*}
Since we started the iteration with $f_1^0(t,x,v)=0$, and $f_2^0(t,x,v)=0$, by induction, we obtain
\begin{align}\label{f1f2est}
\| f_1^n(t) \|_{L^{\infty,q}_{x,v}} &\leq (\|f_0\|_{L^{\infty,q}_{x,v}} + \delta)e^{-(1-\epsilon)t}, \qquad
\| f_2^n(t) \|_{L^{\infty}_{x,v}(\mu^{-\zeta})} \leq C_q\delta e^{-(1-\epsilon)t},
\end{align}
for all $n\geq 0 $. We note that as we proved in Lemma \ref{f1lem}, sequence $\{f_1^n\}_{n\geq 0}$ is the Cauchy sequence. Thus there exists $f_1\in L^{\infty,q}_{x,v}$ such that
\begin{align*}
f_1^n \rightarrow f_1, \qquad \mbox{with} \qquad \| f_1(t) \|_{L^{\infty,q}_{x,v}} &\leq (\|f_0\|_{L^{\infty,q}_{x,v}} + \delta)e^{-(1-\epsilon)t}.
\end{align*}
For sequence $\{f_2^n\}_{n\geq 0}$, we obtained uniform boundedness. Thus there exists a weak star converging subsequence $f_2$ such that
\begin{align*}
f_2^n \overset{\ast}{\rightharpoonup}  f_2, \qquad \mbox{with} \qquad \| f_2(t) \|_{L^{\infty,q}_{x,v}} &\leq C_q\delta e^{-(1-\epsilon)t}.
\end{align*}
Thanks to \eqref{f1f2est}, $f_1^n$ and $f_2^n$ are weakly compact in $L^1$, and 
for any finite time $T$, there exists a positive constant $C_T$ such that
\begin{align*}
\int_0^T\int_{\mathbb{T}^3}\int_{\mathbb{R}^3} |v|^3(\mu+f_1^{n+1}+f_2^n) dvdxdt < C_T.
\end{align*}
This third-moment estimate combined with velocity averaging lemma in \cite{Perthame}, we have the following strong compactness for the macroscopic fields
\begin{align*}
&\int_{\mathbb{R}^3} (\mu+f_1^{n+1}+f_2^n) dv := \rho^n \rightarrow \rho, \\
&\int_{\mathbb{R}^3} v(\mu+f_1^{n+1}+f_2^n) dv := \rho^nU^n \rightarrow \rho U, \\
& \int_{\mathbb{R}^3} |v|^2(\mu+f_1^{n+1}+f_2^n) dv := 3\rho^nT^n+\rho^n|U^n|^2 \rightarrow 3\rho T + \rho|U|^2.
\end{align*}
On the other hand, the weak compactness of $f_1^n$ and $f_2^n$ gives rise to the weak compactness of the local Maxwellian. The weak compactness of $\mathcal{M}(\mu+f_1^{n+1}+f_2^n)$ together with the strong compactness of $(\rho^n,U^n,T^n)$ yields the following desired weak convergence:
\begin{align*}
\mathcal{M}(\mu+f_1^{n+1}+f_2^n) \rightharpoonup \mathcal{M}(\mu+ f_1+f_2),
\end{align*}
in $L^1([0,T]\times \mathbb{T}^3)$. See \cite{Perthame,Yun3} for detailed arguments. This guarantees that $f_1+f_2$ is a solution of the system \eqref{f1f2}.
Moreover, the solution $f_1+f_2$ satisfies
\begin{align*}
\| (f_1+f_2)(t) \|_{L^{\infty,q}_{x,v}} &\leq C_q\delta e^{-(1-\epsilon)t}.
\end{align*}
To prove the conservation laws, we add the two equations in \eqref{f1f2}:
\begin{align*}
\partial_t (f_1^{n+1}+f_2^{n+1}) + v\cdot\nabla_x (f_1^{n+1}+f_2^{n+1}) +  (f_1^{n+1}+f_2^{n+1}) = \mathbf{P}(f_1^{n+1}+f_2^{n+1})+ \Gamma(f_1^{n+1}+f_2^n).
\end{align*}
We take $\Pi$ on both sides to have
\begin{align}\label{conservf1f2}
\partial_t \Pi(f_1^{n+1}+f_2^{n+1}) = \Pi \left(\Gamma(f_1^{n+1}+f_2^n)\right).
\end{align}
Now we show that $\Pi(\Gamma(f))=0$ for any function $f\in L^{\infty,q}_{x,v}$.
For any function $f\in L^{\infty,q}_{x,v}$ with $F=\mu+f$, the cancellation property of the BGK operator \eqref{cancellation} implies
\begin{align*}
\mathbf{P}(\nu (\mathcal{M}(F)-F)) =0.
\end{align*}
Applying the linearization of the BGK operator in Lemma \ref{linearize}, we have
\begin{align*}
0=\mathbf{P}(\nu(\mathcal{M}(F)-F))=\mathbf{P}((\mathbf{P}f-f) +\Gamma(f))=\mathbf{P}(\Gamma(f)),
\end{align*}
which guarantees $\Pi \left(\Gamma(f_1^{n+1}+f_2^n)\right)=0$. This combining with \eqref{conservf1f2}, we have
\begin{align*}
\Pi(f_1^{n+1}(t)+f_2^{n+1}(t)) =\Pi(f_1^{n+1}(0)+f_2^{n+1}(0)) = \Pi f_0 = 0.
\end{align*}
Therefore $f_1^n+f_2^n$ satisfies the conservation laws for all $n\geq0$. \newline
For stability and uniqueness, let $f$ and $\tilde{f}$ be solutions corresponding to the initial data $f_0$ and $\tilde{f}_0$, respectively. Subtracting the following two equations
\begin{align*}
\p_{t}f + v\cdot\nabla_x f  + f &= \mathbf{P}f + \Gamma(f), \cr
\p_{t}\tilde{f} + v\cdot\nabla_x \tilde{f}  + \tilde{f} &= \mathbf{P}\tilde{f} + \Gamma(\tilde{f}),
\end{align*}
yields
\begin{align*}
\p_{t}(f-\tilde{f}) + v\cdot\nabla_x (f-\tilde{f})  + f-\tilde{f} &= \mathbf{P}(f-\tilde{f}) + \Gamma(f)-\Gamma(\tilde{f}).
\end{align*}
Using the semi-group operator, we write
\begin{align*}
(f-\tilde{f})(t,x,v) = S(t)(f_0-\tilde{f}_0)(x-vt,v) +\int_0^t  S(t-s)(\Gamma(f)-\Gamma(\tilde{f}))(s,x-v(t-s),v)ds.
\end{align*}
Multiplying $e^{(1-\epsilon)t}\langle v \rangle^q$ on both sides and applying Lemma \ref{nonlinsubt}, we have
\begin{align*}
\sup_{s\in[0,t]}e^{(1-\epsilon)s}\|(f-\tilde{f})(s)\|_{L^{\infty,q}_{x,v}} &\leq  e^{-\epsilon t}\|f_0-\tilde{f}_0\|_{L^{\infty,q}_{x,v}} +\int_0^t e^{-\epsilon(t-s)} e^{(1-\epsilon)s}\|(\Gamma(f)-\Gamma(\tilde{f}))(s)\|_{L^{\infty,q}_{x,v}}ds \cr
&\leq \|f_0-\tilde{f}_0\|_{L^{\infty,q}_{x,v}} +\frac{1}{\epsilon}C_{\delta} \sup_{s\in[0,t]}e^{(1-\epsilon)s} \|(f-\tilde{f})(s)\|_{L^{\infty,q}_{x,v}}.
\end{align*}
For sufficiently small $C_{\delta}< \epsilon$, we obtain stabtility and uniqueness of the solution. For the non-negativity of the solution, we recover the equation for $F=\mu+f$:
\begin{align*}
\p_{t}F + v\cdot\nabla_x F +\nu F= \nu \mathcal{M}(F).
\end{align*}
Then the mild form of $F$ with non-negativity of $F_0$ and $\mathcal{M}(F)$ gives the non-negativty of the solution
\begin{align*}
F(t,x,v) = e^{-\int_0^t \nu(s,x)ds} F_0(x-vt,v) + \int_0^t e^{-\int_s^t \nu(\tau,x)d\tau}\mathcal{M}(F)(s,x-v(t-s),v)ds \geq 0.
\end{align*}
This completes the proof of Proposition \ref{smallthm}.
\end{proof}

\appendix
\section{Nonlinear part of the BGK operator}\label{Nonlincomp}
In this section, we give the explicit form of the nonlinear term of the BGK model. We compute the second derivative of the BGK operator and specify the polynomial form $\mathcal{P}_{ij}$ and the number $\alpha_{ij}$ and $\beta_{ij}$ satisfying
\begin{align*}
\left[\nabla_{(\rho_{\theta},\rho_{\theta} U_{\theta}, G_{\theta})}^2 \mathcal{M}(\theta)\right]_{ij} &= \frac{\mathcal{P}_{ij}((v-U_{\theta}),U_{\theta},T_{\theta})}{\rho_{\theta}^{\alpha_{ij}}T_{\theta}^{\beta_{ij}}}\mathcal{M}(\theta).
\end{align*}
Because $\mathcal{M}(\theta)$ depends on $(\rho_{\theta},U_{\theta},T_{\theta})$, we need to use the chain rule twice.
For brevity, we omit the $\theta$ dependency in this section. \\
{\bf The first derivative:} We first review the previous computations for the BGK operator:
\begin{lemma}\label{Jaco}\emph{\cite{Yun1}} When $\rho>0$, for the relation $(\rho,U,T)$ and $(\rho,\rho U,G)$ in \eqref{rhoUG}, we have
\begin{align*}
J = \frac{\partial(\rho,\rho U,G)} {\partial(\rho,U,T)}= \left[ {\begin{array}{cccccc}
1 & 0 & 0 & 0 & 0  \\
U_1 & \rho & 0 & 0 & 0 \\
U_2 & 0 & \rho & 0 & 0 \\
U_3 & 0 & 0 & \rho & 0 \\
\frac{3T+|U|^2-3}{\sqrt{6}} & \frac{2\rho U_1}{\sqrt{6}} & \frac{2\rho U_2}{\sqrt{6}} & \frac{2\rho U_3}{\sqrt{6}} & \frac{3\rho }{\sqrt{6}}
\end{array} } \right],
\end{align*}
and
\begin{align*}
J^{-1} = \left(\frac{\partial(\rho,\rho U,G)} {\partial(\rho,U,T)}\right)^{-1}= \left[ {\begin{array}{cccccc}
1 & 0 & 0 & 0 & 0  \\
-\frac{U_1}{\rho} & \frac{1}{\rho} & 0 & 0 & 0 \\
-\frac{U_2}{\rho} & 0 & \frac{1}{\rho} & 0 & 0 \\
-\frac{U_3}{\rho} & 0 & 0 & \frac{1}{\rho} & 0 \\
\frac{|U|^2-3T+3}{3\rho} & -\frac{2}{3}\frac{U_1}{\rho} & -\frac{2}{3}\frac{U_2}{\rho} & -\frac{2}{3}\frac{U_3}{\rho} & \sqrt{\frac{2}{3}}\frac{1}{\rho}
\end{array} } \right].
\end{align*}
The first derivative of the local Maxwellian with respect to the macroscopic fields gives
\begin{align*}
\nabla_{(\rho,U,T)} \mathcal{M}(F) = \left(\frac{1}{\rho},\frac{v-U}{T},\left(-\frac{3}{2}\frac{1}{T}+\frac{|v-U|^2}{2T^2}\right) \right)\mathcal{M}(F).
\end{align*}
\end{lemma}

Now we compute the first derivative of $\mathcal{M}(F)$ with respect to $(\rho,\rho U, G)$. Since the local Maxwellian $\mathcal{M}(F)$ depends on $(\rho,U,T)$, we should apply the following change of variable:
\begin{align}\label{firstderi}
\nabla_{(\rho,\rho U, G)} \mathcal{M}(F) &=\left(\frac{\partial(\rho,\rho U,G)} {\partial(\rho,U,T)}\right)^{-1}\nabla_{(\rho,U,T)}\mathcal{M}(F).
\end{align}
We denote the right-hand-side of \eqref{firstderi} as $R$. Applying Lemma \ref{Jaco}, we have
\begin{align}\label{firstderi2}
\begin{split}
R&=\left[ {\begin{array}{cccccc}
		1 & 0 & 0 & 0 & 0  \\
		-\frac{U_1}{\rho} & \frac{1}{\rho} & 0 & 0 & 0 \\
		-\frac{U_2}{\rho} & 0 & \frac{1}{\rho} & 0 & 0 \\
		-\frac{U_3}{\rho} & 0 & 0 & \frac{1}{\rho} & 0 \\
		\frac{|U|^2-3T+3}{3\rho} & -\frac{2}{3}\frac{U_1}{\rho} & -\frac{2}{3}\frac{U_2}{\rho} & -\frac{2}{3}\frac{U_3}{\rho} & \sqrt{\frac{2}{3}}\frac{1}{\rho}
\end{array} } \right]
\left[ {\begin{array}{c}
\frac{1}{\rho} \\ \frac{v_1-U_1}{T} \\ \frac{v_2-U_2}{T} \\ \frac{v_3-U_3}{T} \\ -\frac{3}{2}\frac{1}{T}+\frac{|v-U|^2}{2T^2}
\end{array} } \right]  \mathcal{M}(F) \cr
&= \left[ {\begin{array}{c}
\frac{1}{\rho} \\ -\frac{U_1}{\rho^2}+\frac{v_1-U_1}{\rho T} \\ -\frac{U_2}{\rho^2}+\frac{v_2-U_2}{\rho T} \\ -\frac{U_3}{\rho^2}+\frac{v_3-U_3}{\rho T} \\ \frac{|U|^2-3T+3}{3\rho^2}-\frac{2}{3}\frac{U \cdot (v-U)}{\rho T}+\sqrt{\frac{2}{3}}\frac{1}{\rho}\left(-\frac{3}{2}\frac{1}{T}+\frac{|v-U|^2}{2T^2}\right)
\end{array} } \right]  \mathcal{M}(F).
\end{split}
\end{align}
{\bf The second derivative:}
Taking $\nabla_{(\rho,\rho U, G)}$ on \eqref{firstderi}, we apply the change of variable $(\rho,\rho U, G)\rightarrow(\rho,U,T)$ once more:
\begin{align}\label{doublederi}
\nabla_{(\rho,\rho U, G)}^2 \mathcal{M}(F) &= \nabla_{(\rho,\rho U, G)} R = \left(\frac{\partial(\rho,\rho U,G)} {\partial(\rho,U,T)}\right)^{-1}\nabla_{(\rho,U,T)}R.
\end{align}
We first calculate $\nabla_{(\rho,U,T)}R$, which is a $5\times5$ matrix. The first column of $\nabla_{(\rho,U,T)}R$ is $(\rho,U,T)$ derivative of the first component of \eqref{firstderi2}, which is $\nabla_{(\rho,U,T)}\left(\frac{1}{\rho}\mathcal{M}(F)\right)$:
\begin{align*}
\nabla_{(\rho,U,T)}\left(\frac{1}{\rho}\mathcal{M}(F)\right) = \left[ {\begin{array}{c}-\frac{1}{\rho^2}\mathcal{M}+\frac{1}{\rho}\partial_{\rho}\mathcal{M} \cr \frac{1}{\rho}\partial_{U_1}\mathcal{M} \cr \frac{1}{\rho}\partial_{U_2}\mathcal{M} \cr \frac{1}{\rho}\partial_{U_3}\mathcal{M} \cr \frac{1}{\rho}\partial_{T}\mathcal{M} \end{array} } \right]  =
\left[ {\begin{array}{c} 0 \cr \frac{v_1-U_1}{\rho T} \cr
\frac{v_2-U_2}{\rho T}
\cr \frac{v_3-U_3}{\rho T}
\cr \left(-\frac{3}{2}\frac{1}{\rho T}+\frac{|v-U|^2}{2\rho T^2}\right) \end{array} } \right]\mathcal{M}(F).
\end{align*}
Similarly, we can compute the second to the fourth column as follows:
\begin{align*}
\nabla_{(\rho,U,T)}\left(\left(-\frac{U_1}{\rho^2}+\frac{v_1-U_1}{\rho T}\right)\mathcal{M}(F)\right) =
\left[ {\begin{array}{c}\frac{U_1}{\rho^3} \cr
-\left(\frac{1}{\rho^2}+\frac{1}{\rho T}\right) + \left(-\frac{U_1}{\rho^2}+\frac{v_1-U_1}{\rho T}\right) \frac{v_1-U_1}{T}  \cr
\left(-\frac{U_1}{\rho^2}+\frac{v_1-U_1}{\rho T}\right)\frac{v_2-U_2}{T} \cr
\left(-\frac{U_1}{\rho^2}+\frac{v_1-U_1}{\rho T}\right)\frac{v_3-U_3}{T} \cr
-\frac{v_1-U_1}{\rho T^2} + \left(-\frac{U_1}{\rho^2}+\frac{v_1-U_1}{\rho T}\right)\left(-\frac{3}{2}\frac{1}{T}+\frac{|v-U|^2}{2T^2}\right)\end{array} } \right]\mathcal{M}(F),
\end{align*}
and
\begin{align*}
\nabla_{(\rho,U,T)}\left(\left(-\frac{U_2}{\rho^2}+\frac{v_2-U_2}{\rho T}\right)\mathcal{M}(F)\right) =
\left[ {\begin{array}{c} \frac{U_2}{\rho^3} \cr
\left(-\frac{U_2}{\rho^2}+\frac{v_2-U_2}{\rho T}\right)\frac{v_1-U_1}{T} \cr
-\left(\frac{1}{\rho^2}+\frac{1}{\rho T}\right)+ \left(-\frac{U_2}{\rho^2}+\frac{v_2-U_2}{\rho T}\right)\frac{v_2-U_2}{T} \cr
\left(-\frac{U_2}{\rho^2}+\frac{v_2-U_2}{\rho T}\right)\frac{v_3-U_3}{T} \cr
-\frac{v_2-U_2}{\rho T^2} + \left(-\frac{U_2}{\rho^2}+\frac{v_2-U_2}{\rho T}\right)\left(-\frac{3}{2}\frac{1}{T}+\frac{|v-U|^2}{2T^2}\right)\end{array} } \right]\mathcal{M}(F),
\end{align*}
and
\begin{align*}
\nabla_{(\rho,U,T)}\left(\left(-\frac{U_3}{\rho^2}+\frac{v_3-U_3}{\rho T}\right)\mathcal{M}(F)\right) =
\left[ {\begin{array}{c}\frac{U_3}{\rho^3} \cr
\left(-\frac{U_3}{\rho^2}+\frac{v_3-U_3}{\rho T}\right)\frac{v_1-U_1}{T}  \cr
\left(-\frac{U_3}{\rho^2}+\frac{v_3-U_3}{\rho T}\right)\frac{v_2-U_2}{T} \cr
-\left(\frac{1}{\rho^2}+\frac{1}{\rho T}\right)+ \left(-\frac{U_3}{\rho^2}+\frac{v_3-U_3}{\rho T}\right)\frac{v_3-U_3}{T} \cr
-\frac{v_3-U_3}{\rho T^2} + \left(-\frac{U_3}{\rho^2}+\frac{v_3-U_3}{\rho T}\right)\left(-\frac{3}{2}\frac{1}{T}+\frac{|v-U|^2}{2T^2}\right)\end{array} } \right]\mathcal{M}(F).
\end{align*}
The fifth column of $\nabla_{(\rho,U,T)}R$ is equal to $\nabla_{(\rho,U,T)}\left( \left(\frac{|U|^2-3T+3}{3\rho^2}-\frac{2}{3}\frac{U \cdot (v-U)}{\rho T}+\sqrt{\frac{2}{3}}\frac{1}{\rho}\left(-\frac{3}{2}\frac{1}{T}+\frac{|v-U|^2}{2T^2}\right)\right) \mathcal{M}(F)\right)$, and which become
\begin{align*}
\left[ {\begin{array}{c} -\frac{|U|^2-3T+3}{3\rho^3} \cr
\frac{2U_1}{3\rho^2}-\frac{2}{3}\frac{v_1-2U_1}{\rho T}-\sqrt{\frac{2}{3}}\frac{1}{\rho}\left(\frac{v_1-U_1}{T^2}\right)  + \frac{v_1-U_1}{T}  \cr
\frac{2U_2}{3\rho^2}-\frac{2}{3}\frac{v_2-2U_2}{\rho T}-\sqrt{\frac{2}{3}}\frac{1}{\rho}\left(\frac{v_2-U_2}{T^2}\right)  + \frac{v_2-U_2}{T} \cr
\frac{2U_3}{3\rho^2}-\frac{2}{3}\frac{v_3-2U_3}{\rho T}-\sqrt{\frac{2}{3}}\frac{1}{\rho}\left(\frac{v_3-U_3}{T^2}\right)  + \frac{v_3-U_3}{T} \cr
-\frac{1}{\rho^2}+\frac{2}{3}\frac{U \cdot (v-U)}{\rho T^2}+\sqrt{\frac{2}{3}}\frac{1}{\rho}\left(\frac{3}{2}\frac{1}{T^2}-\frac{|v-U|^2}{T^3}\right)  +\left(-\frac{3}{2}\frac{1}{T}+\frac{|v-U|^2}{2T^2}\right)\end{array} } \right] \mathcal{M}(F).
\end{align*}
Note that we computed $\nabla_{(\rho,U,T)}R$ in \eqref{doublederi}. Now we should multiply the matrix $\nabla_{(\rho,U,T)}R$ by $\left(\frac{\partial(\rho,\rho U,G)} {\partial(\rho,U,T)}\right)^{-1}$.
It is a product of two $5\times 5$ matrices. We present the calculations for each row. \newline
$\bullet$ The first row of $\nabla_{(\rho,\rho U, G)}^2 \mathcal{M}(F)$:
\begin{align*}
\left[0  ~ \frac{U_1}{\rho^3} ~ \frac{U_2}{\rho^3} ~ \frac{U_3}{\rho^3} ~ -\frac{|U|^2-3T+3}{3\rho^3}\right]\mathcal{M}(F).
\end{align*}
$\bullet$ The second row of $\nabla_{(\rho,\rho U, G)}^2 \mathcal{M}(F)$:
\begin{align*}
\left[ {\begin{array}{c}
\frac{v_1-U_1}{\rho^2 T} \cr -\frac{U_1^2}{\rho^4}-\left(\frac{1}{\rho^3}+\frac{1}{\rho^2 T}\right) + \left(-\frac{U_1}{\rho^3}+\frac{v_1-U_1}{\rho^2 T}\right) \frac{v_1-U_1}{T} \cr -\frac{U_1U_2}{\rho^4} +\left(-\frac{U_2}{\rho^3}+\frac{v_2-U_2}{\rho^2 T}\right)\frac{v_1-U_1}{T} \cr -\frac{U_1U_3}{\rho^4}+\left(-\frac{U_3}{\rho^3}+\frac{v_3-U_3}{\rho^2 T}\right)\frac{v_1-U_1}{T} \cr
\frac{U_1}{\rho}\frac{|U|^2-3T+3}{3\rho^3}+\frac{1}{\rho}\left(
\frac{2U_1}{3\rho^2}-\frac{2}{3}\frac{v_1-2U_1}{\rho T}-\sqrt{\frac{2}{3}}\frac{1}{\rho}\left(\frac{v_1-U_1}{T^2}\right)  + \frac{v_1-U_1}{T}\right)
\end{array} } \right]^T \mathcal{M}(F).
\end{align*}
$\bullet$ The third row of $\nabla_{(\rho,\rho U, G)}^2 \mathcal{M}(F)$:
\begin{align*}
\left[ {\begin{array}{c}
\frac{v_2-U_2}{\rho^2 T} \cr
-\frac{U_1U_2}{\rho^4} +\left(-\frac{U_1}{\rho^3}+\frac{v_1-U_1}{\rho^2 T}\right)\frac{v_2-U_2}{T} \cr
-\frac{U_2^2}{\rho^4}-\left(\frac{1}{\rho^3}+\frac{1}{\rho^2 T}\right) + \left(-\frac{U_2}{\rho^3}+\frac{v_2-U_2}{\rho^2 T}\right) \frac{v_2-U_2}{T} \cr
-\frac{U_2U_3}{\rho^4}+\left(-\frac{U_3}{\rho^3}+\frac{v_3-U_3}{\rho^2 T}\right)\frac{v_2-U_2}{T} \cr
\frac{U_1}{\rho}\frac{|U|^2-3T+3}{3\rho^3}+\frac{1}{\rho}\left(\frac{2U_2}{3\rho^2}-\frac{2}{3}\frac{v_2-2U_2}{\rho T}-\sqrt{\frac{2}{3}}\frac{1}{\rho}\left(\frac{v_2-U_2}{T^2}\right)  + \frac{v_2-U_2}{T}\right)
\end{array} } \right]^T \mathcal{M}(F).
\end{align*}
$\bullet$ The fourth row of $\nabla_{(\rho,\rho U, G)}^2 \mathcal{M}(F)$:
\begin{align*}
\left[ {\begin{array}{c} \frac{v_3-U_3}{\rho^2 T} \cr
-\frac{U_1U_3}{\rho^4} +\left(-\frac{U_1}{\rho^3}+\frac{v_1-U_1}{\rho^2 T}\right)\frac{v_3-U_3}{T} \cr
-\frac{U_2U_3}{\rho^4}+\left(-\frac{U_2}{\rho^3}+\frac{v_2-U_2}{\rho^2 T}\right)\frac{v_3-U_3}{T} \cr
-\frac{U_3^2}{\rho^4}-\left(\frac{1}{\rho^3}+\frac{1}{\rho^2 T}\right) + \left(-\frac{U_3}{\rho^3}+\frac{v_3-U_3}{\rho^2 T}\right) \frac{v_3-U_3}{T} \cr
\frac{U_3}{\rho}\frac{|U|^2-3T+3}{3\rho^3}+\frac{1}{\rho}\left(\frac{2U_3}{3\rho^2}-\frac{2}{3}\frac{v_3-2U_3}{\rho T}-\sqrt{\frac{2}{3}}\frac{1}{\rho}\left(\frac{v_3-U_3}{T^2}\right)  + \frac{v_3-U_3}{T}\right)
\end{array} } \right]^T \mathcal{M}(F).
\end{align*}
$\bullet$ The fifth row of $\nabla_{(\rho,\rho U, G)}^2 \mathcal{M}(F)$:
Since the fifth row is too complicated, we present each component separately.
The first component of the fifth row, i.e. $[\nabla_{(\rho,\rho U, G)}^2 \mathcal{M}(F)]_{51}$:
\begin{align*}
\left(-\frac{2}{3}\frac{U}{\rho}\cdot\frac{v-U}{\rho T}
+\sqrt{\frac{2}{3}}\frac{1}{\rho}\left(-\frac{3}{2}\frac{1}{\rho T}+\frac{|v-U|^2}{2\rho T^2}\right)\right)\mathcal{M}(F).
\end{align*}
The second component of the fifth row, i.e. $[\nabla_{(\rho,\rho U, G)}^2 \mathcal{M}(F)]_{52}$:
\begin{align*}
\bigg[\frac{|U|^2-3T+3}{3\rho}\frac{U_1}{\rho^3} -\frac{2}{3}\frac{U}{\rho}\cdot\frac{v-U}{T}\left(\left(-\frac{U_1}{\rho^2}+\frac{v_1-U_1}{\rho T}\right) \right) -\frac{2}{3}\frac{U_1}{\rho}\left(-\left(\frac{1}{\rho^2}+\frac{1}{\rho T}\right)\right)  \cr
+\sqrt{\frac{2}{3}}\frac{1}{\rho}\left(-\frac{v_1-U_1}{\rho T^2} + \left(-\frac{U_1}{\rho^2}+\frac{v_1-U_1}{\rho T}\right)\left(-\frac{3}{2}\frac{1}{T}+\frac{|v-U|^2}{2T^2}\right)\right)\bigg]\mathcal{M}(F).
\end{align*}
The third component of the fifth row, i.e. $[\nabla_{(\rho,\rho U, G)}^2 \mathcal{M}(F)]_{53}$:
\begin{align*}
\bigg[\frac{|U|^2-3T+3}{3\rho}\frac{U_2}{\rho^3} -\frac{2}{3}\frac{U}{\rho}\cdot \frac{v-U}{T} \left(-\frac{U_2}{\rho^2}+\frac{v_2-U_2}{\rho T}\right)
-\frac{2}{3}\frac{U_2}{\rho}\left(-\left(\frac{1}{\rho^2}+\frac{1}{\rho T}\right)\right) \cr
+ \sqrt{\frac{2}{3}}\frac{1}{\rho}\left(-\frac{v_2-U_2}{\rho T^2} + \left(-\frac{U_2}{\rho^2}+\frac{v_2-U_2}{\rho T}\right)\left(-\frac{3}{2}\frac{1}{T}+\frac{|v-U|^2}{2T^2}\right)\right)\bigg]\mathcal{M}(F).
\end{align*}
The fourth component of the fifth row, i.e. $[\nabla_{(\rho,\rho U, G)}^2 \mathcal{M}(F)]_{54}$:
\begin{align*}
\bigg[\frac{|U|^2-3T+3}{3\rho}\frac{U_3}{\rho^3} -\frac{2}{3}\frac{U}{\rho}\cdot \frac{v-U}{T} \left(-\frac{U_3}{\rho^2}+\frac{v_3-U_3}{\rho T}\right)
-\frac{2}{3}\frac{U_3}{\rho}\left(-\left(\frac{1}{\rho^2}+\frac{1}{\rho T}\right)\right) \cr
+ \sqrt{\frac{2}{3}}\frac{1}{\rho}\left(-\frac{v_3-U_3}{\rho T^2} + \left(-\frac{U_3}{\rho^2}+\frac{v_3-U_3}{\rho T}\right)\left(-\frac{3}{2}\frac{1}{T}+\frac{|v-U|^2}{2T^2}\right)\right)\bigg]\mathcal{M}(F).
\end{align*}
The fifth component of the fifth row, i.e. $[\nabla_{(\rho,\rho U, G)}^2 \mathcal{M}(F)]_{55}$:
\begin{align*}
\bigg[\frac{|U|^2-3T+3}{3\rho}\left(-\frac{|U|^2-3T+3}{3\rho^3}\right)
-\frac{2}{3}\frac{U}{\rho}\cdot \left(\frac{2U}{3\rho^2}-\frac{2}{3}\frac{v-2U}{\rho T}-\sqrt{\frac{2}{3}}\frac{1}{\rho}\left(\frac{v-U}{T^2}\right)  + \frac{v-U}{T}\right)  \cr
+ \sqrt{\frac{2}{3}}\frac{1}{\rho}\left(-\frac{1}{\rho^2}+\frac{2}{3}\frac{U \cdot (v-U)}{\rho T^2}+\sqrt{\frac{2}{3}}\frac{1}{\rho}\left(\frac{3}{2}\frac{1}{T^2}-\frac{|v-U|^2}{T^3}\right)  +\left(-\frac{3}{2}\frac{1}{T}+\frac{|v-U|^2}{2T^2}\right)\right)\bigg]\mathcal{M}(F).
\end{align*}
We completed the calculation of $5\times5$ matrix $\nabla_{(\rho,\rho U, G)}^2 \mathcal{M}(F)$.  \\

\noindent {\bf Acknowledgement:}
G.-C. Bae is supported by the National Research Foundation of Korea(NRF) grant funded by the Korea government(MSIT) (No. 2021R1C1C2094843). G.-H. Ko and D.-H. Lee are supported by the National Research Foundation of Korea(NRF) grant funded by the Korean government(MSIT)(No. NRF-2019R1C1C1010915). S.-B. Yun is supported by Samsung Science and Technology Foundation under Project Number SSTF-BA1801-02.

\hide
\providecommand{\bysame}{\leavevmode\hbox to3em{\hrulefill}\thinspace}
\providecommand{\MR}{\relax\ifhmode\unskip\space\fi MR }
\providecommand{\MRhref}[2]{%
  \href{http://www.ams.org/mathscinet-getitem?mr=#1}{#2}
}
\providecommand{\href}[2]{#2}

\unhide


\begin{thebibliography}{10}

\bibitem{BGCY}
Gi-Chan Bae and Seok-Bae Yun, \emph{Stationary quantum {BGK} model for bosons
  and fermions in a bounded interval}, J. Stat. Phys. \textbf{178} (2020),
  no.~4, 845--868. \MR{4064205}

\bibitem{BY2023}
Gi-Chan Bae and Seok-Bae Yun \emph{The {S}hakhov model near a global {M}axwellian}, Nonlinear Anal.
  Real World Appl. \textbf{70} (2023), Paper No. 103742, 33. \MR{4498741}

\bibitem{Bang-Y}
Jeaheang Bang and Seok-Bae Yun, \emph{Stationary solutions for the ellipsoidal
  {BGK} model in a slab}, J. Differential Equations \textbf{261} (2016),
  no.~10, 5803--5828. \MR{3548271}

\bibitem{Bello}
A.~Bellouquid, \emph{Global existence and large-time behavior for {BGK} model
  for a gas with non-constant cross section}, Transport Theory Statist. Phys.
  \textbf{32} (2003), no.~2, 157--184. \MR{1969901}

\bibitem{BGK}
P.~L. Bhatnagar, E.~P. Gross, and M.~Krook, \emph{A model for collision
  processes in gases. i. small amplitude processes in charged and neutral
  one-component systems}, Phys. Rev. \textbf{94} (1954), 511--525.

\bibitem{Guo-Briant}
Marc Briant and Yan Guo, \emph{Asymptotic stability of the {B}oltzmann equation
  with {M}axwell boundary conditions}, J. Differential Equations \textbf{261}
  (2016), no.~12, 7000--7079. \MR{3562318}

\bibitem{Brull-Y}
Stephane Brull and Seok-Bae Yun, \emph{Stationary flows of the es-bgk model
  with the correct prandtl number}, arXiv preprint arXiv:2012.08490 (2020).

\bibitem{CaoJFA}
Chuqi Cao, \emph{Cutoff {B}oltzmann equation with polynomial perturbation near
  {M}axwellian}, J. Funct. Anal. \textbf{283} (2022), no.~9, Paper No. 109641,
  105. \MR{4458223}

\bibitem{CKLVPB}
Yunbai Cao, Chanwoo Kim, and Donghyun Lee, \emph{Global strong solutions of the
  {V}lasov-{P}oisson-{B}oltzmann system in bounded domains}, Arch. Ration.
  Mech. Anal. \textbf{233} (2019), no.~3, 1027--1130. \MR{3961294}

\bibitem{C}
Carlo Cercignani, \emph{The {B}oltzmann equation and its applications}, Applied
  Mathematical Sciences, vol.~67, Springer-Verlag, New York, 1988. \MR{1313028}

\bibitem{CIP}
Carlo Cercignani, Reinhard Illner, and Mario Pulvirenti, \emph{The mathematical
  theory of dilute gases}, Applied Mathematical Sciences, vol. 106,
  Springer-Verlag, New York, 1994. \MR{1307620}

\bibitem{CK1}
Hongxu Chen and Chanwoo Kim, \emph{Regularity of stationary {B}oltzmann
  equation in convex domains}, Arch. Ration. Mech. Anal. \textbf{244} (2022),
  no.~3, 1099--1222. \MR{4419612}

\bibitem{DMOS}
Jean Dolbeault, Peter Markowich, Dietmar Oelz, and Christian Schmeiser,
  \emph{Non linear diffusions as limit of kinetic equations with relaxation
  collision kernels}, Arch. Ration. Mech. Anal. \textbf{186} (2007), no.~1,
  133--158. \MR{2338354}

\bibitem{DHWY2017}
Renjun Duan, Feimin Huang, Yong Wang, and Tong Yang, \emph{Global
  well-posedness of the {B}oltzmann equation with large amplitude initial
  data}, Arch. Ration. Mech. Anal. \textbf{225} (2017), no.~1, 375--424.
  \MR{3634029}

\bibitem{DKL2020}
Renjun Duan, Gyounghun Ko, and Donghyun Lee, \emph{The boltzmann equation with
  large-amplitude initial data and specular reflection boundary condition},
  arXiv preprint arXiv:2011.01503 (2020).

\bibitem{DW2019}
Renjun Duan and Yong Wang, \emph{The {B}oltzmann equation with large-amplitude
  initial data in bounded domains}, Adv. Math. \textbf{343} (2019), 36--109.
  \MR{3880824}

\bibitem{DWY2017}
Renjun Duan, Yong Wang, and Tong Tang, \emph{Global existence for the ellipsoidal
  bgk model with initial large oscillations}, SCIENTIA SINICA Mathematica
  \textbf{47} (2017), no.~10, 1143--1154.

\bibitem{F-J}
Francis Filbet and Shi Jin, \emph{An asymptotic preserving scheme for the
  {ES}-{BGK} model of the {B}oltzmann equation}, J. Sci. Comput. \textbf{46}
  (2011), no.~2, 204--224. \MR{2753243}

\bibitem{GKTT2}
Y.~Guo, C.~Kim, D.~Tonon, and A.~Trescases, \emph{B{V}-regularity of the
  {B}oltzmann equation in non-convex domains}, Arch. Ration. Mech. Anal.
  \textbf{220} (2016), no.~3, 1045--1093. \MR{3466841}

\bibitem{GuoQAM}
Yan Guo, \emph{Bounded solutions for the {B}oltzmann equation}, Quart. Appl.
  Math. \textbf{68} (2010), no.~1, 143--148. \MR{2598886}

\bibitem{Guo10}
Yan Guo, \emph{Decay and continuity of the {B}oltzmann equation in bounded
  domains}, Arch. Ration. Mech. Anal. \textbf{197} (2010), no.~3, 713--809.
  \MR{2679358}

\bibitem{GKTT2017}
Yan Guo, Chanwoo Kim, Daniela Tonon, and Ariane Trescases, \emph{Regularity of
  the {B}oltzmann equation in convex domains}, Invent. Math. \textbf{207}
  (2017), no.~1, 115--290. \MR{3592757}

\bibitem{HY}
Byung-Hoon Hwang and Seok-Bae Yun, \emph{Stationary solutions to the boundary
  value problem for the relativistic {BGK} model in a slab}, Kinet. Relat.
  Models \textbf{12} (2019), no.~4, 749--764. \MR{3984749}

\bibitem{Issau}
D.~Issautier, \emph{Convergence of a weighted particle method for solving the
  {B}oltzmann ({B}.{G}.{K}.) equation}, SIAM J. Numer. Anal. \textbf{33}
  (1996), no.~6, 2099--2119. \MR{1427455}

\bibitem{Kim2011}
Chanwoo Kim, \emph{Formation and propagation of discontinuity for {B}oltzmann
  equation in non-convex domains}, Comm. Math. Phys. \textbf{308} (2011),
  no.~3, 641--701. \MR{2855537}

\bibitem{KimLee}
Chanwoo Kim and Donghyun Lee, \emph{The {B}oltzmann equation with specular
  boundary condition in convex domains}, Comm. Pure Appl. Math. \textbf{71}
  (2018), no.~3, 411--504. \MR{3762275}

\bibitem{cylinder}
Chanwoo Kim and Donghyun Lee, \emph{Decay of the {B}oltzmann equation with the specular boundary
  condition in non-convex cylindrical domains}, Arch. Ration. Mech. Anal.
  \textbf{230} (2018), no.~1, 49--123. \MR{3840911}

\bibitem{KLKRM}
Gyounghun Ko and Donghyun Lee, \emph{On $ c^{2} $ solution of the
  free-transport equation in a disk}, Kinetic and Related Models \textbf{16}
  (2023), no.~3, 311--372.

\bibitem{KLP2022}
Gyounghun Ko, Donghyun Lee, and Kwanghyuk Park, \emph{The large amplitude
  solution of the {B}oltzmann equation with soft potential}, J. Differential
  Equations \textbf{307} (2022), 297--347. \MR{4337824}

\bibitem{Mellet}
Antoine Mellet, \emph{Fractional diffusion limit for collisional kinetic
  equations: a moments method}, Indiana Univ. Math. J. \textbf{59} (2010),
  no.~4, 1333--1360. \MR{2815035}

\bibitem{M-M-M}
Antoine Mellet, St\'{e}phane Mischler, and Cl\'{e}ment Mouhot, \emph{Fractional
  diffusion limit for collisional kinetic equations}, Arch. Ration. Mech. Anal.
  \textbf{199} (2011), no.~2, 493--525. \MR{2763032}

\bibitem{M}
L.~Mieussens, \emph{Convergence of a discrete-velocity model for the
  {B}oltzmann-{BGK} equation}, Comput. Math. Appl. \textbf{41} (2001), no.~1-2,
  83--96. \MR{1808507}

\bibitem{M-S}
Luc Mieussens and Henning Struchtrup, \emph{Numerical comparison of
  bhatnagar--gross--krook models with proper prandtl number}, Physics of Fluids
  \textbf{16} (2004), no.~8, 2797--2813.

\bibitem{Nouri}
Anne Nouri, \emph{An existence result for a quantum {BGK} model}, Math. Comput.
  Modelling \textbf{47} (2008), no.~3-4, 515--529. \MR{2378854}

\bibitem{Perthame}
B.~Perthame, \emph{Global existence to the {BGK} model of {B}oltzmann
  equation}, J. Differential Equations \textbf{82} (1989), no.~1, 191--205.
  \MR{1023307}

\bibitem{P-P}
B.~Perthame and M.~Pulvirenti, \emph{Weighted {$L^\infty$} bounds and
  uniqueness for the {B}oltzmann {BGK} model}, Arch. Rational Mech. Anal.
  \textbf{125} (1993), no.~3, 289--295. \MR{1245074}

\bibitem{F-R}
Giovanni Russo and Francis Filbet, \emph{Semilagrangian schemes applied to
  moving boundary problems for the {BGK} model of rarefied gas dynamics},
  Kinet. Relat. Models \textbf{2} (2009), no.~1, 231--250. \MR{2472158}

\bibitem{RSY}
Giovanni Russo, Pietro Santagati, and Seok-Bae Yun, \emph{Convergence of a
  semi-{L}agrangian scheme for the {BGK} model of the {B}oltzmann equation},
  SIAM J. Numer. Anal. \textbf{50} (2012), no.~3, 1111--1135. \MR{2970736}

\bibitem{RY}
Giovanni Russo and Seok-Bae Yun, \emph{Convergence of a semi-{L}agrangian
  scheme for the ellipsoidal {BGK} model of the {B}oltzmann equation}, SIAM J.
  Numer. Anal. \textbf{56} (2018), no.~6, 3580--3610. \MR{3892427}

\bibitem{SR2}
Laure Saint-Raymond, \emph{Discrete time {N}avier-{S}tokes limit for the {BGK}
  {B}oltzmann equation}, Comm. Partial Differential Equations \textbf{27}
  (2002), no.~1-2, 149--184. \MR{1886958}

\bibitem{SR1}
Laure Saint-Raymond, \emph{From the {BGK} model to the {N}avier-{S}tokes equations}, Ann.
  Sci. \'{E}cole Norm. Sup. (4) \textbf{36} (2003), no.~2, 271--317.
  \MR{1980313}

\bibitem{Sone}
Yoshio Sone, \emph{Kinetic theory and fluid dynamics}, Modeling and Simulation
  in Science, Engineering and Technology, Birkh\"{a}user Boston, Inc., Boston,
  MA, 2002. \MR{1919070}

\bibitem{Sone2}
Yoshio Sone, \emph{Molecular gas dynamics}, Modeling and Simulation in Science,
  Engineering and Technology, Birkh\"{a}user Boston, Inc., Boston, MA, 2007,
  Theory, techniques, and applications. \MR{2274674}

\bibitem{Ukai}
Seiji Ukai, \emph{Stationary solutions of the {BGK} model equation on a finite
  interval with large boundary data}, Proceedings of the {F}ourth
  {I}nternational {W}orkshop on {M}athematical {A}spects of {F}luid and
  {P}lasma {D}ynamics ({K}yoto, 1991), vol.~21, 1992, pp.~487--500.
  \MR{1194459}

\bibitem{Wal}
Pierre Welander, \emph{On the temperature jump in a rarefied gas}, Ark. Fys.
  \textbf{7} (1954), 507--553. \MR{62041}

\bibitem{Yun1}
Seok-Bae Yun, \emph{Cauchy problem for the {B}oltzmann-{BGK} model near a
  global {M}axwellian}, J. Math. Phys. \textbf{51} (2010), no.~12, 123514, 24.
  \MR{2779616}

\bibitem{Yun2}
Seok-Bae Yun, \emph{Classical solutions for the ellipsoidal {BGK} model with fixed
  collision frequency}, J. Differential Equations \textbf{259} (2015), no.~11,
  6009--6037. \MR{3397316}

\bibitem{Yun3}
Seok-Bae Yun, \emph{Ellipsoidal {BGK} model near a global {M}axwellian}, SIAM J.
  Math. Anal. \textbf{47} (2015), no.~3, 2324--2354. \MR{3357626}

\bibitem{Yun22}
Seok-Bae Yun, \emph{Ellipsoidal {BGK} model for polyatomic molecules near
  {M}axwellians: a dichotomy in the dissipation estimate}, J. Differential
  Equations \textbf{266} (2019), no.~9, 5566--5614. \MR{3912760}

\bibitem{Zhang}
Xianwen Zhang, \emph{On the {C}auchy problem of the {V}lasov-{P}oisson-{BGK}
  system: global existence of weak solutions}, J. Stat. Phys. \textbf{141}
  (2010), no.~3, 566--588. \MR{2728846}

\bibitem{Z-H}
Xianwen Zhang and Shigeng Hu, \emph{{$L^p$} solutions to the {C}auchy problem
  of the {BGK} equation}, J. Math. Phys. \textbf{48} (2007), no.~11, 113304,
  17. \MR{2370243}

\end{thebibliography}
\end{document}